\documentclass[11pt,reqno,final]{amsart}

\usepackage[scaled]{helvet} 
\usepackage{fourier} 
\usepackage[usenames,dvipsnames,svgnames,table]{xcolor}
\usepackage{color}
\usepackage{url,graphicx,tabularx,array,geometry, amsthm, amsfonts,parskip,alltt, amsmath,tikz}
\usepackage{hyperref}

\usepackage{url,graphicx,tabularx,array,geometry, amsthm, amsfonts,parskip,alltt, amsmath,hyperref,xcolor}
\usepackage{mathtools,tikz,algorithm,algpseudocode,amssymb}
\usepackage{pgf,tikz}
\usepackage{mathrsfs}
\usepackage{epstopdf}
\usetikzlibrary{arrows}

\newtheorem{theorem}{Theorem}

\usetikzlibrary{arrows}

\newtheorem{lemma}{Lemma}

\DeclareMathOperator{\argmin}{argmin}

\DeclareMathOperator{\R}{\mathbb{R}}

\newcommand{\ve}[1]{\mathbf{#1}}

\makeatletter
\renewcommand*{\ALG@name}{Method}
\makeatother

\title{Randomized Projection Methods for Linear Systems with Arbitrarily Large Sparse Corruptions}

\author{
Jamie Haddock and Deanna Needell
}

\begin{document}
\begin{abstract}
In applications like medical imaging, error correction, and sensor networks, one needs to solve large-scale linear systems that may be corrupted by a small number of arbitrarily large corruptions. We consider solving such large-scale systems of linear equations $A\ve{x}=\ve{b}$ that are inconsistent due to corruptions in the measurement vector $\ve{b}$. With this as our motivating example, we develop an approach for this setting that allows detection of the corrupted entries and thus convergence to the ``true'' solution of the original system. We provide analytical justification for our approaches as well as experimental evidence on real and synthetic systems.
\end{abstract}

\thanks{This material is based upon work supported by the National Science Foundation under Grant No. DMS-1440140 while the authors were in residence at the Mathematical Sciences Research Institute in Berkeley, California, during the Fall 2017 semester. JH was also partially supported by NSF grant DMS-1522158 and the University of California, Davis Dissertation Fellowship. DN was also supported by NSF CAREER award $\#1348721$ and NSF BIGDATA $\#1740325$.}

\maketitle

\section{Introduction}
We consider solving large-scale systems of linear equations represented by a matrix $A\in\mathbb{R}^{m\times n}$ and vector $\ve{b}\in\mathbb{R}^m$. We are interested in the highly overdetermined setting, where $m \gg n$, which means the system need not necessarily have a solution.  One may then seek the least squares solution $\ve{x}_{\text{LS}}$ which minimizes $\|A\ve{x}-\ve{b}\|^2$ (where $\|\cdot\|$ denotes the Euclidean norm); many efficient solvers have been developed that converge to such a solution. An alternative setting is one where there is a solution $\ve{x}^*$ (which we refer to as the \textit{pseudo-solution}) to our desired system $A\ve{x}=\ve{b}^*$, but rather than observing $\ve{b}^*$ we only have access to a corrupted version, $\ve{b}$, where $\ve{b} = \ve{b}^* + \ve{b}_C$.  When the number of non-zero entries in $\ve{b}_C$, denoted $\|\ve{b}_C\|_0$, is small relative to $m$, one may still hope to recover the ``true'' solution $\ve{x}^*$.\footnote{This paper extends work previously presented in \cite{HN17Corrupted}.} This type of sparse corruption models many applications, ranging from medical imaging to sensor networks and error correcting codes.  For example, a small number of sensors may malfunction, resulting in large catastrophic reporting errors in the vector $\ve{b}$; since the reporting errors themselves may be arbitrarily large, the least squares solution is far from the desired solution, but since the number of such reporting errors is small, we may still hope to recover the true solution to the uncorrupted system. We emphasize that such a pseudo-solution $\ve{x}^*$ may be very far from the least squares solution $\ve{x}_{\text{LS}}$ when the entries in $\ve{b}_C$ are large, even when there are only a few non-zero corruptions; see Figure \ref{LSpic} for a visual. Similar types of sparse errors may also appear in medical imaging from artifacts or system malfunctions, or in error correcting codes from transmission errors.  Indeed, the problem of so-called \textit{sparse recovery} is well-studied in the approximation and compressed sensing literature \cite{FR12:Mathematical-Introduction,eldar2012compressed}.  However, in this paper, we are concerned with the setting where the system is highly \textit{overdetermined}, the errors in $\ve{b}$ are sparse and large, and the system may be so large-scale that it cannot be fully loaded into memory. This latter property has sparked a recent resurgence of work in the area of iterative solvers that do not need access to the entire system at once \cite{gordon1975image,herman1978relaxation,natterer2001mathematics,strohmer2009randomized}. 
Our work is motivated by such iterative methods.

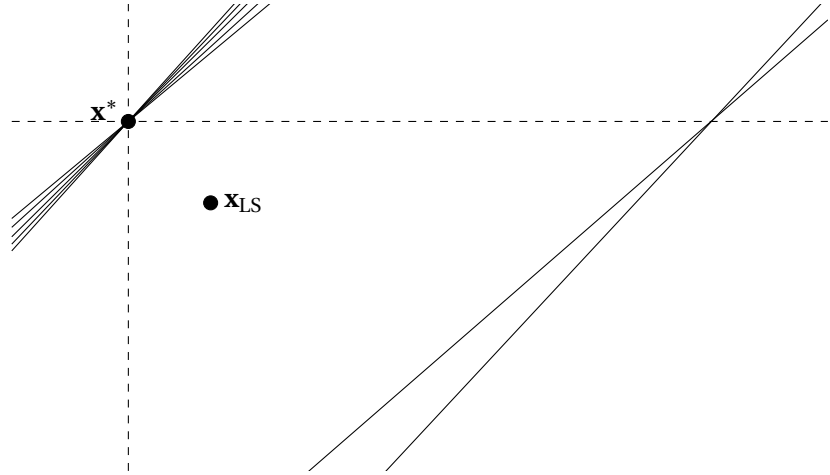
\begin{figure}
\begin{center}
\newcommand{\MyPath}{(-2,13/7)--(9/11,16/11)--(-1.32558, -1.04651)--(-2., -0.625)}
\begin{tikzpicture}[scale=1.55]
\clip (-1,1) rectangle (6,-3);
\draw (-1,-.909)--(6,5.454);
\draw (-1,-1.111)--(6,6.666);
\draw (-1,-.990)--(6,5.940);
\draw (-1,-1.053)--(6,6.318);
\draw (-1,-.833)--(6,4.998);
\draw (-1,-6.451)--(6,1.074);
\draw (-1,-5.218)--(6,.872);
\draw[dashed] (0,7)--(0,-7);
\draw[dashed] (-1,0)--(6,0);
\draw[fill] (0, 0) circle [radius=0.06];
\draw[fill] (.7058, -0.6988) circle [radius=0.06];
\node at (-0.2,0.1) {$\ve{x}^{*}$};
\node at (0.99,-0.7) {$\ve{x}_{\text{LS}}$};
\end{tikzpicture}
\end{center}
\caption{A system for which the pseudo-solution $\ve{x}^\star$ is very far from the least squares solution $\ve{x}_{\text{LS}}$.  Lines represent the hyperplanes consisting of all systems $\{\ve{x}: \ve{a}_i^T\ve{x} = b_i\}$ for rows $\ve{a}_i^T$ of $A$.}\label{LSpic}
\end{figure}

It is important to point out that solving for the pseudo-solution of systems $A\ve{x} = \ve{b} = \ve{b}^* + \ve{b}_C$ where $\|\ve{b}_C\|_0$ is small relative to $m$ is related to finding a solution of a large consistent system within an inconsistent system. 
 The problem of finding the maximal consistent subsystem of an inconsistent system is known as \texttt{MAX-FS} and it is known to be NP-hard without a polynomial-time approximation scheme (PTAS) \cite{amaldikann}.  This problem is one of the focuses of infeasibility analysis, the study of changes necessary to make an infeasible system of linear constraints feasible \cite{murty2000infeasibility}. There are approximation algorithms \cite{nutovreichman} or, of course, one can solve the problem in a brute force manner.  Approaches for solving this problem generally fall into two categories, heuristic methods which use solutions to relaxations of subproblems, and branch-and-cut strategies for solving the integer program formulation of this problem \cite{chinneck2001fast, mangasarian1994misclassification, pfetsch2008branch}.  These methods are not row-action methods and generally require operating on the entire system or large subsystems, making them impractical for our setting. It has been previously observed that the behavior of projection and relaxation methods can detect inconsistent systems and thermal variants of these methods have been developed for identifying consistent subsystems and even for approximating \texttt{MAX-FS} solutions \cite{cloninger, amaldi}.  These methods are row-action methods, but are designed for \texttt{MAX-FS} problems rather than the large consistent subsystem setting we have described, and a comparison is not natural.
 
From this viewpoint, our problem is that of solving a consistent \textit{subsystem} of equations where we assume the size of this system is large relative to the size of the entire inconsistent system.  We are motivated by the setting in which one must solve an overdetermined system of equations in which few of the rows have been corrupted.  Often, in applications, one is not concerned with finding the maximal feasible subsystem, but instead finding its solution, and this subsystem can be assumed to be large.  Applications in this framework include logic programming, error detection and correction in telecommunications, and infeasible linear programming models.  Our approach to solving for $\ve{x}^\star$ will make use of the randomized Kaczmarz method, which we discuss next.

The Kaczmarz method is a popular iterative solver for overdetermined systems of linear equations and is especially preferred for systems with an extremely large number of rows.  The method consists of sequential orthogonal projections toward the solution set of a single equation (or subsystem).  Given the system $A\ve{x} = \ve{b}$, the method computes iterates by projecting onto the hyperplane defined by the equation $\ve{a}_i^T \ve{x} = b_i$ where $\ve{a}_i^T$ is a selected row of the matrix $A$ and $b_i$ is the corresponding entry of $\ve{b}$.  The iterates are recursively defined as 
\begin{equation}\label{rk}
\ve{x}_{k+1} = \ve{x}_k + \frac{b_i - \ve{a}_i^T\ve{x}}{\|\ve{a}_i\|^2} \ve{a}_i
\end{equation} 
where $\ve{a}_i$ is selected from among the rows of $A$.  The seminal work \cite{strohmer2009randomized} proved exponential convergence for the randomized Kaczmarz method where the rows $\ve{a}_i$ are chosen with probability $\|\ve{a}_i\|^2/\|A\|_F^2$.  Since then many variants of the method have been proposed and analyzed for various types of systems, see e.g. \cite{gower2015randomized,needell2014paved,REK,eldar2011acceleration,LL10:Randomized-Methods}.

It is known that the randomized Kaczmarz method converges for systems $A\ve{x} = \ve{b}$ corrupted by noise with an error threshold dependent on $A$ and the noise. In \cite{deanna} it was shown that this method has iterates that satisfy:
\begin{equation}\label{RKits}
\|\ve{x}_k - \ve{x}_{\text{LS}}\|^2 \leq \left(1 - \frac{\sigma_{\text{min}}^2(A)}{\|A\|_F^2}\right)^k \|\ve{x}_0 -  \ve{x}_{\text{LS}}\|^2 + \frac{\|A\|_F^2}{\sigma_{\text{min}}^2(A)}\|\ve{e}\|_{\infty}^2,
\end{equation}
where $\sigma_{\text{min}}(A)$ denotes the minimum singular value of $A$, $\|A\|_F$ its Frobenius norm, $ \ve{x}_{\text{LS}}$ the least squares solution and $\ve{e} = \ve{b} - A \ve{x}_{\text{LS}}$ denotes the error term (also known as the residual).
 There are variants of this method that converge to the least squares solution \cite{censor1983strong,REK}, however these typically either require operations on the columns or unknown relaxation parameters.  Additionally, it is known that if a linear system of equations or inequalities is feasible then randomized Kaczmarz will provide a \emph{proof} or \emph{certificate of feasibility}, and there are probabilistic guarantees on how quickly it will do so \cite{SKM}.  However, we are now interested in using randomized Kaczmarz for infeasible systems in which the least-squares solution is unsatisfactory because it is far from satisfying most of the equations (e.g. the noise is sparse and large).  
 
\subsection{Contribution}

We develop methods that seek to identify the corrupted entries in $\ve{b}$ and then converge to the pseudo-solution.  Our methods consist of several `rounds' of many iterations of the Randomized Kaczmarz (RK) method.  The intuition behind these methods is that if there are only few corrupted equations and many consistent equations, then the iterations of RK will select consistent equations with high probability, producing an iterate near the pseudo-solution and then the largest residual entries will correspond to the corrupted equations. We give a lower bound on the probability that a single round detects the corrupted equations.  One may run many independent rounds and increase the probability of detecting these corrupted constraints.  We then give a lower bound on the probability that one of these many rounds will detect the corrupted equations.

\subsection{Notation}

For simplicity, we define some general notation to be used throughout the paper.  Let $||\cdot||$ refer to the Euclidean norm.  We will denote vectors in boldface (e.g., $\ve{x}$), and matrices and scalars in non-bold (e.g., $A$ and $b_i$).  We use $\ve{a}_i^T \in \mathbb{R}^n$ to represent the $i$th row of $A \in \mathbb{R}^{m \times n}$ and $\ve{e}_i \in \mathbb{R}^m$ to represent the $i$th coordinate vector.  We will denote the origin as $\ve{0} \in \mathbb{R}^n$. Let $[m] = \{1, 2, ..., m\}$ and let $[A]$ refer to the set of indices of the rows of matrix $A$ (i.e. for $A \in \mathbb{R}^{m \times n}$, $[A] = [m]$).  For $D \subset [A]$, we let $A_{D^C} = A_{[A] - D}$ be the submatrix of $A$ whose rows are indexed by the complement of $D$.  Denote the minimum singular value of $A$ as $\sigma_{\min}(A)$.

As mentioned above, we consider the situation in which $A \in \mathbb{R}^{m \times n}$ and $\ve{b} \in \mathbb{R}^m$ define an inconsistent system of equations, but there is a large consistent subsystem.  That is, $A$ and $\ve{b}^* \in \mathbb{R}^m$ defines the consistent system of equations with solution $\ve{x}^* \in \mathbb{R}^n$ (i.e. $A\ve{x}^* = \ve{b}^*$) which we will refer to as the pseudo-solution of the system of equations defined by $A$ and $\ve{b}$.  The right hand side vector $\ve{b}^*$ is corrupted by $\ve{b}_C$, so $\ve{b} = \ve{b}^* + \ve{b}_C$.  Let $I \subset [m]$ be the set of indices of inconsistent equations, i.e. $\text{supp}(\ve{b}_C) = I$ and $s:= |I| \ll m$.   
We refer to the amount of corruption in each index of $I$ by $\epsilon_i \in \mathbb{R}$, so $\ve{b}_C = \sum_{i \in I} \epsilon_i \ve{e}_i$. We let $\epsilon^*$ be the smallest absolute entry of the corruption, $\epsilon^* := \min_{i \in I} |\epsilon_i|$.  We will also use $A_*$ to refer to the matrix $A$ without the rows indexed by $I$, $A_* = A_{I^C}$, and likewise for $\ve{b}_*$. Note then that $\ve{b}_* := \ve{b}_{I^C} = \ve{b}^*_{I^C}$.   
For convenience, we will assume throughout the paper that the rows of $A$ are normalized to have unit norm.

\section{A Kaczmarz-type approach for corrupted systems}

We consider here solving a consistent system of linear equations that has been corrupted, $A\ve{x} = \ve{b}^* + \ve{b}_C$ with $\|\ve{b}_C\|_0 = s \ll m$. Formally, given matrix $A$ and right hand side vector $\ve{b}$, we are searching for $\ve{x}^*$ given by:
\begin{equation}\label{eq:zeronormopt}
(\ve{b}_C, \ve{x}^*) = \argmin_{\ve{b}_C, \ve{x}} \|\ve{b}_C\|_0 \quad\text{such that}\quad A\ve{x} = \ve{b}-\ve{b}_C. 
\end{equation}
One can design pathological examples of corrupted linear systems in which the solution to \eqref{eq:zeronormopt} differs from the pseudo-solution; however, typically these solutions coincide. 
In particular, for $\ve{a}_i$ in general position, this holds.
First, we recall the RK method and fundamental convergence results.

\begin{algorithm}
\caption{Randomized Kaczmarz \cite{strohmer2009randomized}}\label{alg:RK}
\begin{algorithmic}[1]
\Procedure{RK}{$A,\ve{b},\ve{x}_0,k$}
\For{$j=1,2,...,k$}
\State $\ve{x}_j = \ve{x}_{j-1} + \frac{b_{i_j} - \ve{a}_{i_j}^T \ve{x}_{j-1}}{\|\ve{a}_{i_j}\|^2} \ve{a}_{i_j}$ where $i_j = t \in [m]$ with probability proportional to $\|\ve{a}_{t}\|^2$.
\EndFor
\State \textbf{return} $\ve{x}_k$
\EndProcedure
\end{algorithmic}
\end{algorithm}
 
 \begin{theorem}\label{strohmervershynin}\cite{strohmer2009randomized}
 Let $\ve{x}$ be the solution of $A\ve{x} = \ve{b}$, then randomized Kaczmarz converges to $\ve{x}$ in expectation with the average error $$\mathbb{E}\|\ve{x}_k - \ve{x}\|^2 \le \left(1 - \frac{\sigma_{\min}^2(A)}{\|A\|_F^2 }\right)^k \|\ve{x}_0 - \ve{x}\|^2.$$
 \end{theorem}

The intuition behind our proposed approach is simple. Since the number of corruptions is small, most iterates of an RK approach will be close to the pseudo-solution, since it is rare to project onto a corrupted hyperplane. Therefore, if we run the RK method several times, or for several \textit{rounds} of iterations, most of the iterates upon which we halt will be close to the pseudo-solution.  Such iterates will also have the property that the largest components of their residual, $|A\ve{x}_k - \ve{b}|,$ will correspond to the large corrupted entries. We can thus utilize this knowledge to gradually detect the corruptions, remove them from the system, and solve for the desired pseudo-solution.

Our proposed methods can thus be described as follows.  Each method consists of $W$ rounds of $k$ RK iterations beginning with $\ve{x}_0 = \ve{0}$.  In each round, we collect the $d$ indices of the largest magnitude residual entries and after all rounds, we solve the system without the rows of $A$ indexed by these collected indices (there may be as many as $dW$ rows removed).  The methods differ in two ways.  First, we can choose to remove $d$ rows within each round (resulting in Method \ref{WKwR} below), or simply collect these indices and remove all collected rows after the $W$ rounds (resulting in Methods \ref{WKwoR} and \ref{WKwoRUS} below).  Second, when waiting to remove the rows until after $W$ rounds, we may simply select the $d$ largest residual entries in each round (Method \ref{WKwoR}), or we may require that the selected indices are always unique (so exactly $dW$ rows are removed), resulting in Method \ref{WKwoRUS}.  The values $W, k$ and $d$ are all parameters of the methods.  We give theoretical results for various natural choices of these parameters.

\begin{algorithm}
\caption{Multiple Round Kaczmarz with Removal}\label{WKwR}
\begin{algorithmic}[1]
\Procedure{MRKwR}{$A,\ve{b},k,W,d$}
\State $B = A$, $\ve{c} = \ve{b}$
\For{$i = 1,2,...W$}
\State $\ve{x}_k^i = RK(B,\ve{c},\ve{0},k)$
\State $D = \text{argmax}_{D \subset [B], |D| = d} \sum_{j \in D} |B\ve{x}_k^i - \ve{c}|_j.$
\State $B = B_{D^C}$, $\ve{c} = \ve{c}_{D^C}$
\EndFor
\State \textbf{return} $\ve{x}$, where $B \ve{x} = \ve{c}$
\EndProcedure
\end{algorithmic}
\end{algorithm}

\begin{algorithm}
\caption{Multiple Round Kaczmarz without Removal}\label{WKwoR}
\begin{algorithmic}[1]
\Procedure{MRKwoR}{$A,\ve{b},k,W,d$}
\State $S = \emptyset$
\For{$i = 1,2,...W$}
\State $\ve{x}_k^i = RK(A,\ve{b},\ve{0},k)$
\State $D = \text{argmax}_{D \subset [A], |D| = d} \sum_{j \in D} |A\ve{x}_k^i - \ve{b}|_j.$
\State $S = S \cup D$
\EndFor
\State \textbf{return} $\ve{x}$, where $A_{S^C} \ve{x} = \ve{b}_{S^C}$
\EndProcedure
\end{algorithmic}
\end{algorithm}

\begin{algorithm}
\caption{Multiple Round Kaczmarz without Removal with Unique Selection}\label{WKwoRUS}
\begin{algorithmic}[1]
\Procedure{MRKwoRUS}{$A,\ve{b},k,W,d$}
\State $S = \emptyset$
\For{$i = 1,2,...W$}
\State $\ve{x}_k^i = RK(A,\ve{b},\ve{0},k)$
\State $D = \text{argmax}_{D \subset [A] - S, |D| = d} \sum_{j \in D} |A\ve{x}_k^i - \ve{b}|_j.$
\State $S = S \cup D$
\EndFor
\State \textbf{return} $\ve{x}$, where $A_{S^C} \ve{x} = \ve{b}_{S^C}$
\EndProcedure
\end{algorithmic}
\end{algorithm}


\subsubsection{Main Results}

Our theoretical results provide a lower bound for the probability of successfully removing all corrupted equations after performing Method \ref{WKwoR} or Method \ref{WKwoRUS} with natural values for $k, d$ and $W$.  Lemma \ref{lem:detectionhorizon} shows that there is a \emph{detection horizon} around the pseudo-solution, so that if $\|\ve{x} - \ve{x}^*\|$ is sufficiently small, the largest residual entries (of $|A\ve{x}-\ve{b}|$) correspond exactly to the corrupted equations and we may distinguish these equations from the consistent system.  Lemma \ref{lem:probbound} gives a value of $k$ so that after $k$ iterations of Randomized Kaczmarz, one can give a nonzero lower bound on the probability that the current iterate is within the detection horizon.  Theorems \ref{thm:d=Iwindowedprob} and \ref{thm:dsmallerthanIwindowedprob} then give lower bounds on the probability of successfully detecting all corrupted equations in one out of all $W$ rounds for Methods \ref{WKwoR} and \ref{WKwoRUS}, respectively.  Proofs of all results are contained in the appendix. 

\begin{lemma}\label{lem:detectionhorizon} If $\|\ve{x} - \ve{x}^*\| < \frac{1}{2} \epsilon^*$ we have that the $d \le s$ indices of largest magnitude residual entries are contained in $I$; that is for $$D = \underset{D \subset [A], |D| = d}{\text{argmax}} \; \underset{i \in D}{\sum} |A\ve{x} - \ve{b}|_i$$ we have $D \subset I$.
\end{lemma}

\begin{lemma}\label{lem:probbound}
Let $0 < \delta < 1$.  Define $$k^* = \max\Bigg(0, \Bigg\lceil \frac{\log\Big(\frac{\delta(\epsilon^*)^2}{4\|\ve{x}^*\|^2}\Big)}{\log\Big(1 - \frac{\sigma_{\min}^2(A_*)}{m-s}\Big)} \Bigg\rceil\Bigg).$$  Then in round $i$
of Method \ref{WKwoR} or Method \ref{WKwoRUS}, the iterate produced by the RK iterations, $\ve{x}_{k^*}^i$ satisfies 
\begin{equation}\label{probbb}
\mathbb{P}\Big[\|\ve{x}_{k^*}^i - \ve{x}^*\| \le \frac{1}{2}\epsilon^*\Big] \ge (1-\delta)\Big(\frac{m-s}{m}\Big)^{k^*}.
\end{equation}
\end{lemma}

First, note that we must restrict $k^*$ to be nonnegative; since $\log\bigg(1 - \frac{\sigma_{\min}^2(A_*)}{m-s}\bigg)$ is negative, if $\log\bigg(\frac{\delta(\epsilon^*)^2}{4\|\ve{x}^*\|^2}\bigg)$ is positive, we must define $k^* = 0$.  However, this corresponds to the situation in which $\epsilon^* > 2 \|\ve{x}^*\|$ and the initial iterate $\ve{x}_0 = \ve{0}$ is within the detection horizon. 
Additionally, note that $k^*$ depends upon $\delta$, so one is not able to make this probability as large as one likes.  As $\delta$ decreases, $k^*$ increases, so the right hand side of \eqref{probbb} is bounded away from $1$.  In Figure \ref{fig:gaussian}, we plot $k^*$ and $(1-\delta)\Big(\frac{m-s}{m}\Big)^{k^*}$ for Gaussian systems with various number of corruptions. In the plots, we see that the value of $\delta$ which maximizes this probability depends upon $s$.  Determining this maximizing $\delta$ was not computable in closed form.  Additionally, we point out that the empirical behavior of the method does not appear to depend upon $\delta$; we believe this is an artifact of our proof.

\begin{theorem}\label{thm:d=Iwindowedprob}
Let $0 < \delta < 1$.  Suppose $d \ge s$, $W \le \lfloor\frac{m-n}{d}\rfloor$ and $k^*$ is as given in Lemma \ref{lem:probbound}.  Then Method \ref{WKwoR} on $A,\ve{b}$ will detect the corrupted equations ($I \subset S$) and the remaining equations given by $A_{[m]-S}, \ve{b}_{[m]-S}$ will have solution $\ve{x}^*$ with probability at least $$1 - \bigg[1 - (1-\delta)\Big(\frac{m-s}{m}\Big)^{k^*}\bigg]^{W}.$$
\end{theorem}

In Figure \ref{fig:gaussian}, we plot $1 - \bigg[1 - (1-\delta)\Big(\frac{m-s}{m}\Big)^{k^*}\bigg]^{W}$ for corrupted Gaussian systems and choices of $\delta$.  Here $W = \lfloor (m-n)/d \rfloor$ and $d = s$.  Again, we reiterate that we believe the dependence upon $\delta$ is an artifact of the proof of Lemma \ref{lem:probbound}.  Substituting e.g., $\delta = 0.5$ in probability bounds gives a value not far from its maximum for all systems we studied; see Figures \ref{fig:gaussian} and \ref{fig:correlated}.

\begin{theorem}\label{thm:dsmallerthanIwindowedprob}
Let $0 < \delta < 1$.  Suppose $d \ge 1$, $W \le \lfloor \frac{m-n}{d} \rfloor$ and $k^*$ is as given in Lemma \ref{lem:probbound}.  Then Method \ref{WKwoRUS} on $A,\ve{b}$ will detect the corrupted equations ($I \subset S$) and the remaining equations given by $A_{[m]-S}, \ve{b}_{[m]-S}$ will have solution $\ve{x}^*$ with probability at least $$1 - \sum_{j=0}^{\lceil s/d \rceil - 1} {W \choose j} p^j (1-p)^{W-j}$$ where $p = (1-\delta)\Big(\frac{m-s}{m}\Big)^{k^*}$.
\end{theorem}
We are not able to prove a result similar to Theorem \ref{thm:d=Iwindowedprob} or Theorem \ref{thm:dsmallerthanIwindowedprob} for Method \ref{WKwR} due to the fact that rounds of this method are not independent because one removes equations after each round.

In Figure \ref{fig:WKwoRUS}, we plot $1 - \sum_{j=0}^{\lceil s/d \rceil - 1} {W \choose j} p^j (1-p)^{W-j}$ for corrupted Gaussian systems and choices of $\delta$.  Here $W = 2$, $d = \lceil s/2 \rceil$, and $k^*$ is as given in Lemma \ref{lem:probbound}.  We believe that the dependence upon $\delta$ is an artifact of our proof.  Evidence suggesting this is seen in the middle and right plots of Figure \ref{fig:WKwoRUS}, as the empirical behavior of Method \ref{WKwoRUS} does not appear to depend upon $\delta$.

These bounds on the probability of successfully detecting all corrupted equations in one round, while provable and nonzero, are pessimistic and do not resemble the experimental rate of success for any systems we studied; see Figures \ref{fig:gaussian} and \ref{fig:correlated}.  A tighter bound on the rate of convergence for particular systems could provide a tighter lower bound on this probability.

\section{Experimental Results}\label{sec:experiments}

We are only able to prove theoretical results when the rounds of Methods \ref{WKwoR} and \ref{WKwoRUS} are independent and for the specified values of $k^*$ and $d$.  However, in practice, these methods perform well for different values of $k$ and $d$, and Method \ref{WKwR} can be quite successful.  In this section, we present experimental results demonstrating the performance of these methods, for various choices of $d$ and $k$, on Gaussian, correlated, and real systems.

We plot our theoretical bounds as well as comparable empirical measures, which we denote `success rates.'  Note that Theorem \ref{thm:d=Iwindowedprob} provides a bound on the probability that in \emph{one} of the rounds of Methods \ref{WKwoR} we successfully detect \emph{all} of the corrupted equations.  For this reason, in Figures \ref{fig:gaussian} and \ref{fig:correlated}, we plot the empirical rate at which our method selects all of the $s$ corrupted equations in one of the $W$ rounds over 100 trials.  However, in Figures \ref{fig:gaussiantotal} and \ref{fig:correlatedtotal}, we plot the rates at which all of the $s$ corrupted equations are selected over \emph{all} of the $W$ rounds over 100 trials, which is a more practical measure of success.  However, Theorem \ref{thm:dsmallerthanIwindowedprob} presents a bound on the probability that all of the $s$ corrupted equations are selected after \emph{all} of the $W$ rounds in Method \ref{WKwoRUS}.  Figure \ref{fig:WKwoRUS} plots this bound alongside the corresponding empirical rate.  The measure of success plotted in each figure is defined in figure caption.

\subsection{Random Data Experiments}

\begin{figure}
		\includegraphics[width=0.33\textwidth]{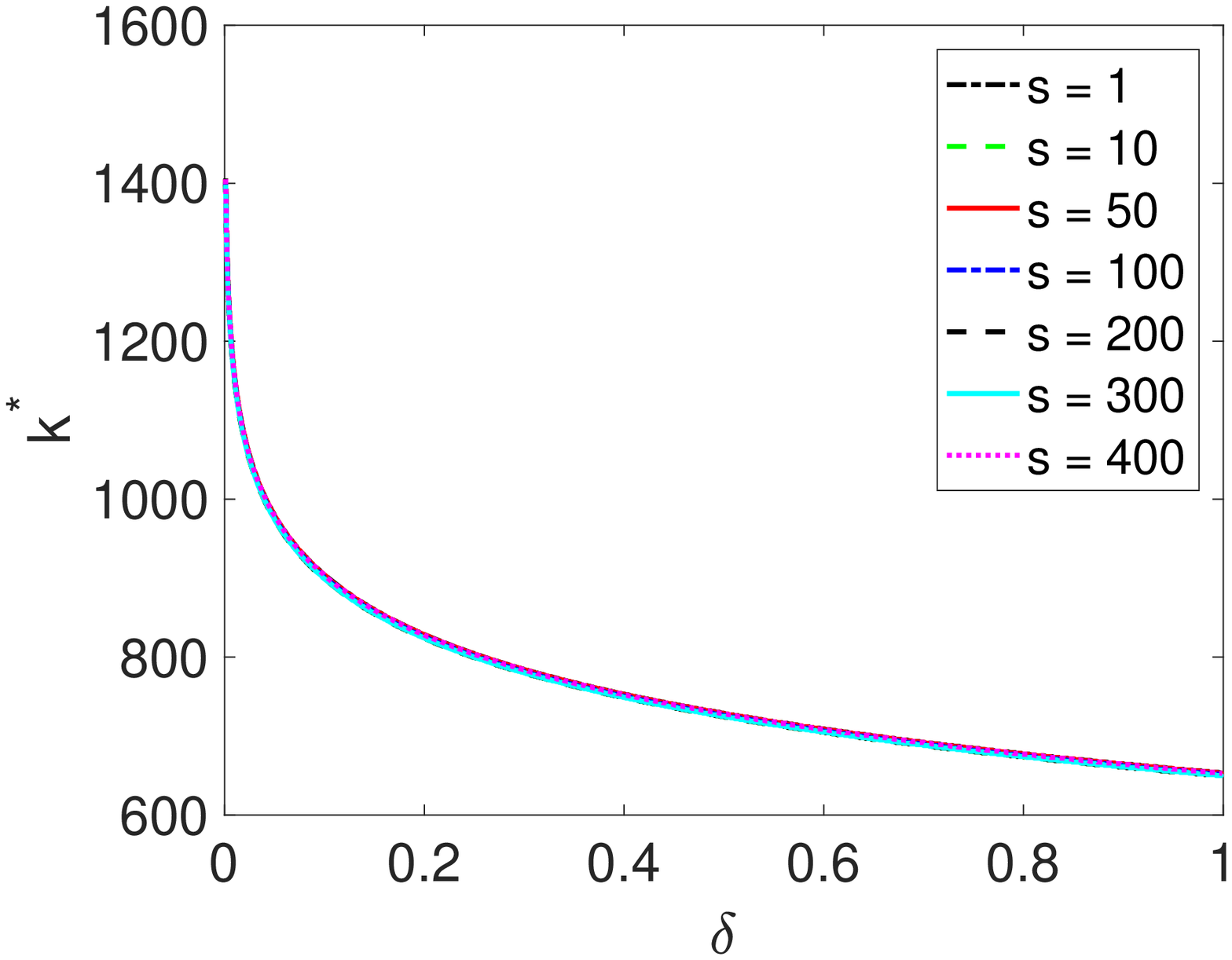} \includegraphics[width=0.33\textwidth]{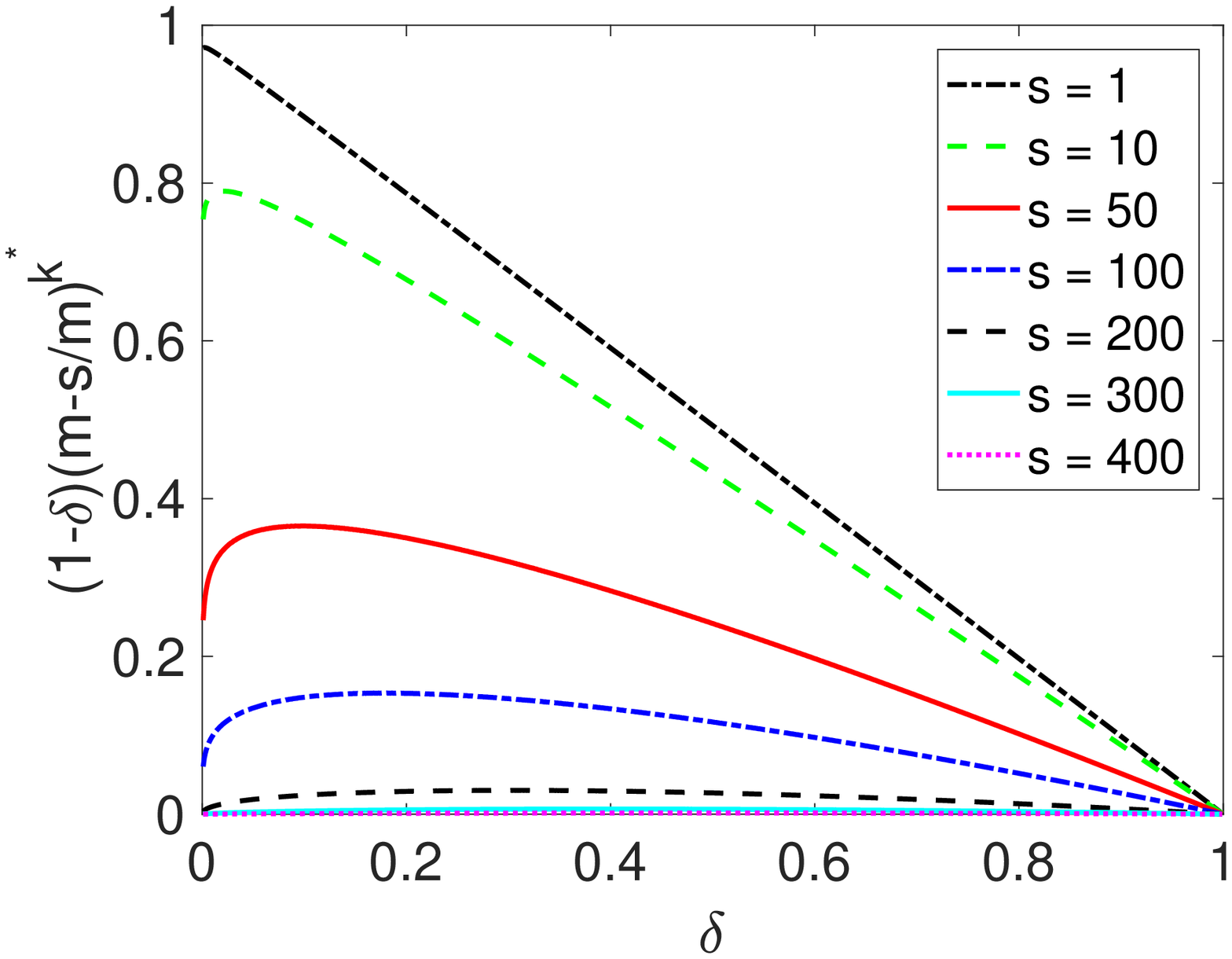}
		\includegraphics[width=0.33\textwidth]{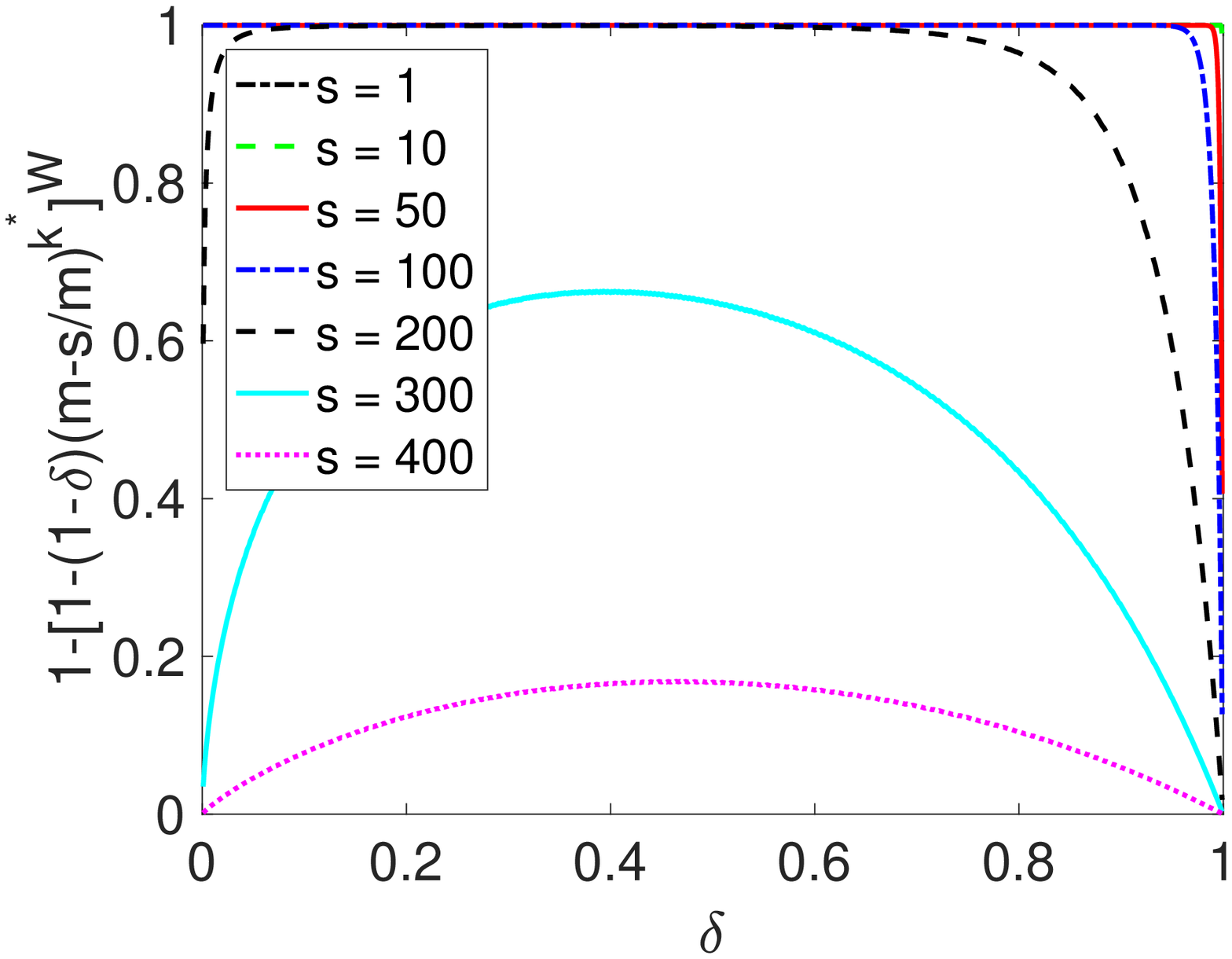} \\\includegraphics[width=0.49\textwidth]{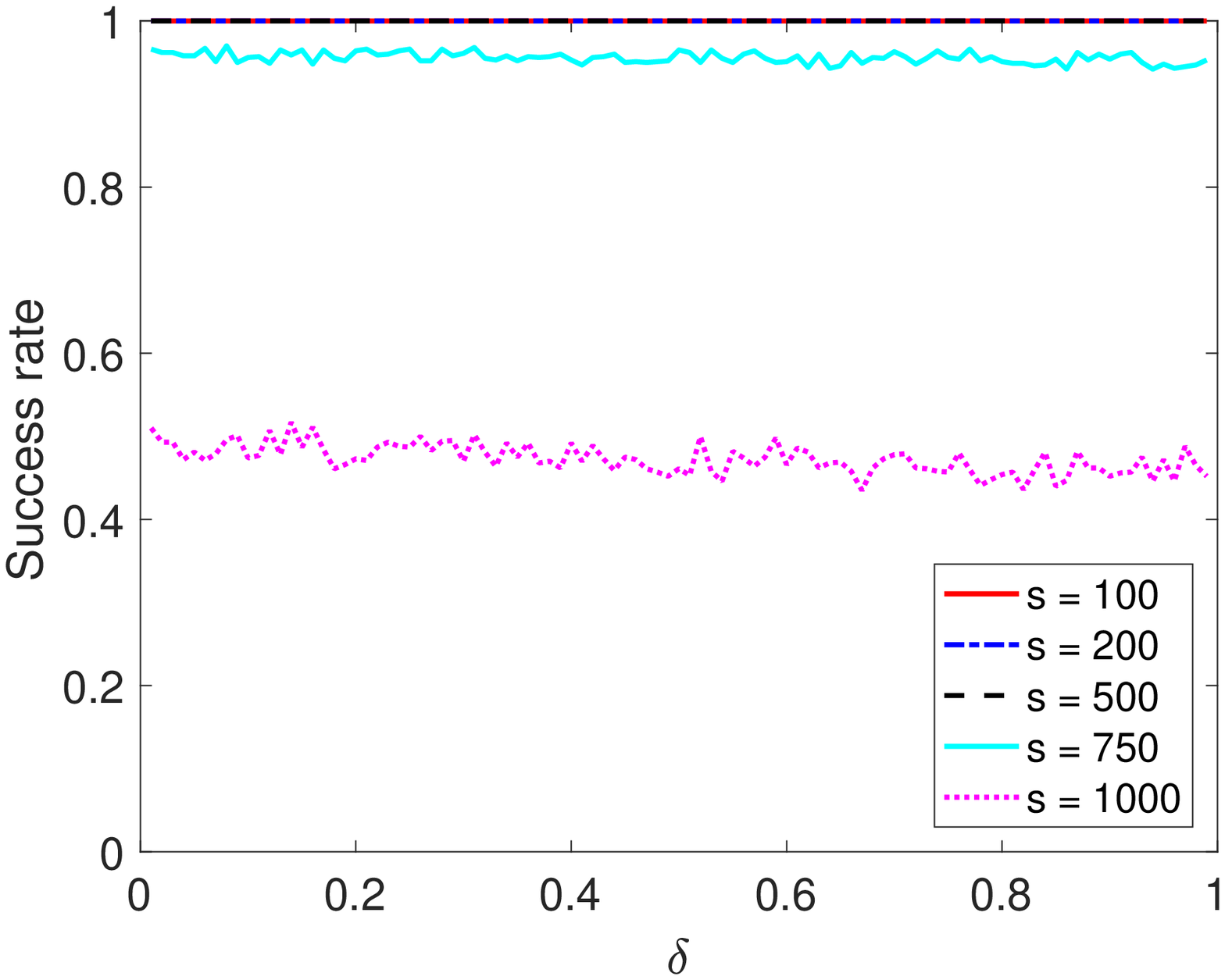}
		\includegraphics[width=0.49\textwidth]{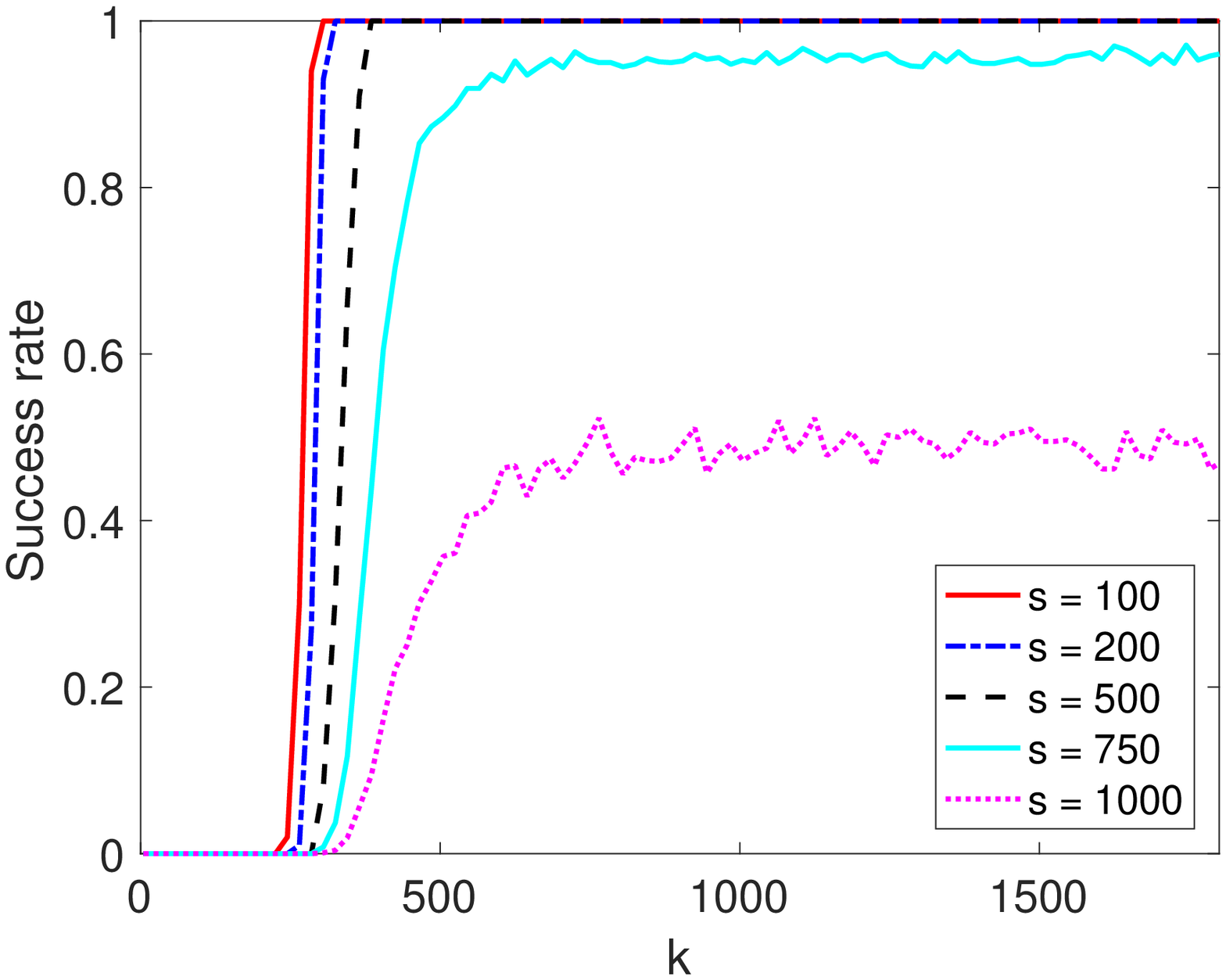}
		\caption{Plots for Method \ref{WKwoR} on $50000 \times 100$ Gaussian system (normalized) with various number of corrupted equations, $s$.  Upper left: $k^*$ as given in Lemma \ref{lem:probbound}; upper middle: Lower bound on probability of successfully detecting all corrupted equations in single round as given in Lemma \ref{lem:probbound}; upper right: Lower bound on probability of successfully detecting all corrupted equations in one round out of $W = \lfloor \frac{m-n}{s} \rfloor$ rounds as given in Theorem \ref{thm:d=Iwindowedprob};  lower left: Experimental ratio of success of detecting all $s$ corrupted equations in one round out of $W = \lfloor \frac{m-n}{s} \rfloor$ for choice of $\delta$; lower right: Experimental ratio of success of detecting all $s$ corrupted equations in one round out of $W = \lfloor \frac{m-n}{s} \rfloor$ for choice of $k$ (number of RK iterations per round).}
		\label{fig:gaussian}
\end{figure}

\begin{figure}
	\begin{center}
		\includegraphics[width=0.49\textwidth]{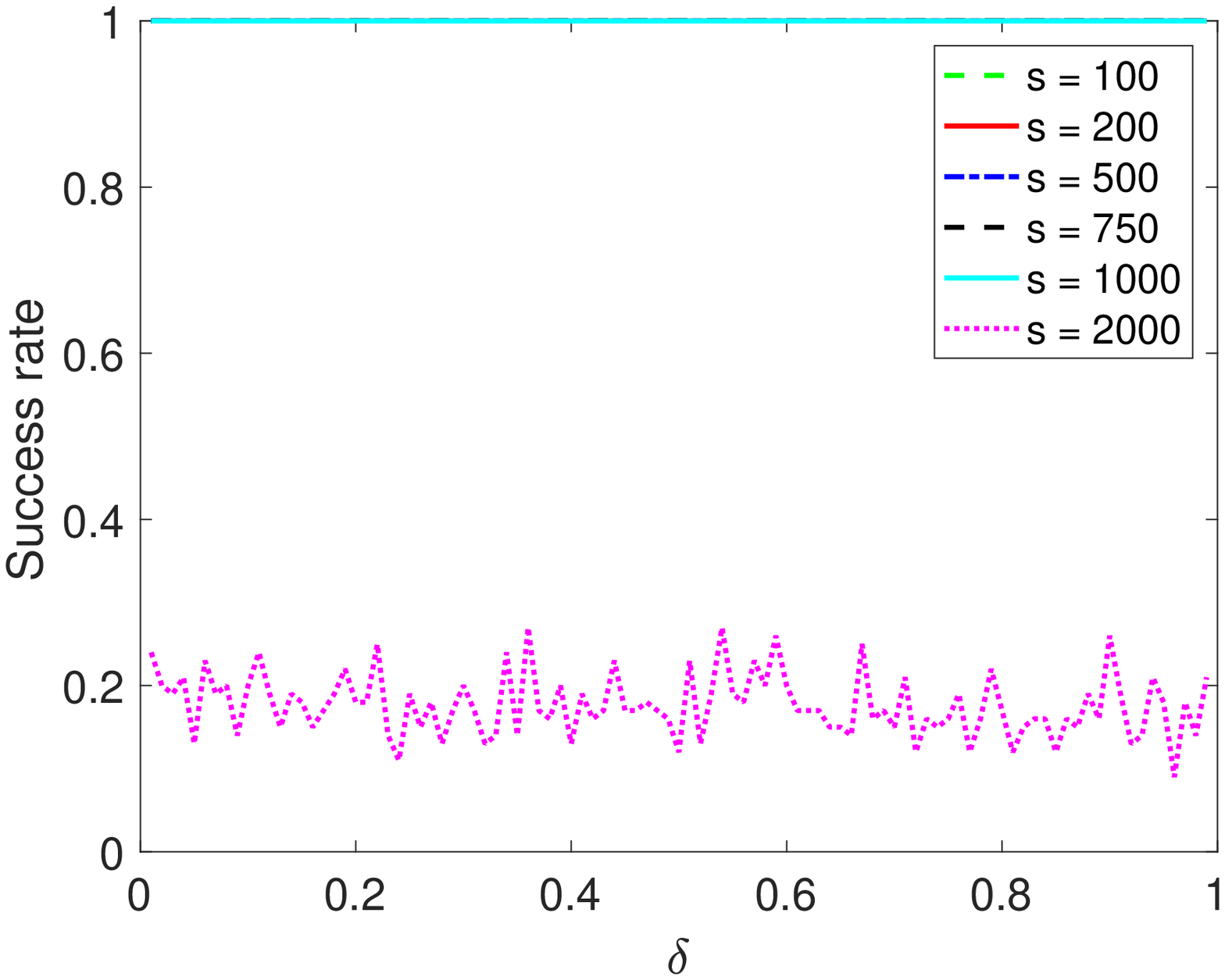} \includegraphics[width=0.49\textwidth]{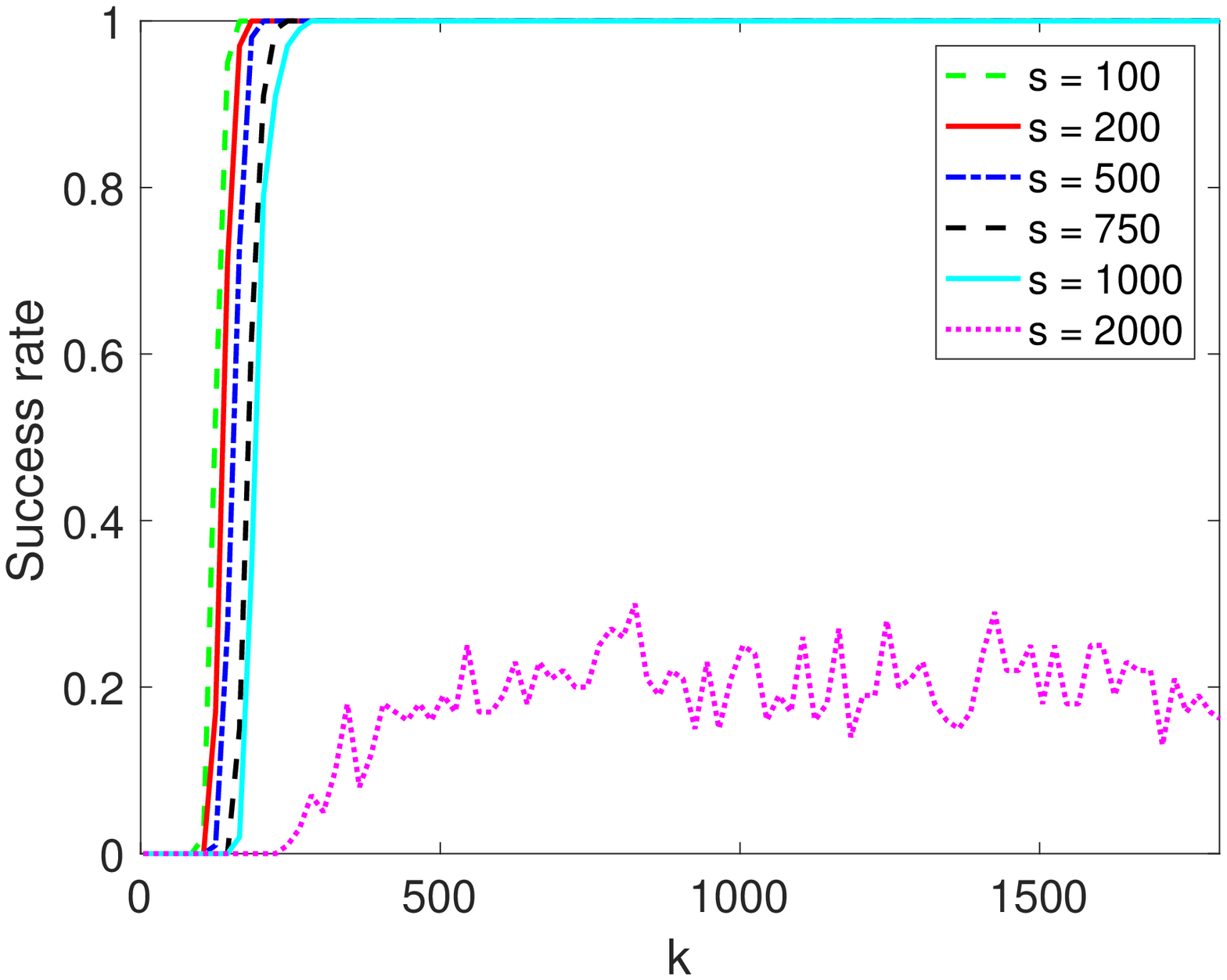}
	\end{center}
	\caption{Plots for Method \ref{WKwoR} on $50000 \times 100$ Gaussian system (normalized) with various number of corrupted equations, $s$. Left: Experimental ratio of successfully detecting all $s$ corrupted equations after all $W = \lfloor \frac{m-n}{s} \rfloor$ rounds for choice of $\delta$; right: Experimental ratio of successfully detecting all $s$ corrupted equations after all $W = \lfloor \frac{m-n}{s} \rfloor$ rounds for choice of $k$ (number of RK iterations per round).}
	\label{fig:gaussiantotal}
\end{figure}

The plots in Figures \ref{fig:gaussian} and \ref{fig:gaussiantotal} are all for Method \ref{WKwoR} on a $50000 \times 100$ Gaussian system defined by $A$ with $a_{ij} \sim \mathcal{N}(0,1)$, then normalized.  The system is corrupted in randomly selected right-hand side entries with random integers in $[1,5]$ so that $\epsilon^* = 1$.  For these plots and experiments, $d=s$.  The upper left image of Figure \ref{fig:gaussian} plots the $k^*$ values defined in Lemma \ref{lem:probbound} for this system, and the upper middle image plots the theoretically guaranteed probability of selecting all $s$ corrupted equations in a single round.  The upper right image of Figure \ref{fig:gaussian} plots the theoretically guaranteed probability of selecting all $s$ corrupted equations in one round out of the $W = \lfloor \frac{m-n}{s}\rfloor$, while the lower left image plots the ratio of successful trials, in which all $s$ corrupted equations were selected in one round of the $W$, out of 100 trials.  The lower right plot of Figure \ref{fig:gaussian} plots how this ratio changes as the number of RK iterations, $k$, in each round varies.  Finally, Figure \ref{fig:gaussiantotal} plots the ratio of successful trials, in which all $s$ corrupted equations were selected after all $W = \lfloor \frac{m-n}{s}\rfloor$ rounds, out of 100 trials as one varies $\delta$ (left) and $k$ (right).  We note that the lower bounds on the probability of successfully detecting all corrupted equations in one round are quite pessimistic; experimentally (in the lower left plot) we see that Method \ref{WKwoR} is able to detect all corruption for much larger numbers of corrupted equations, $s$, than predicted theoretically (in the upper right plot).  Additionally, we note that experimentally, successfully detecting the corrupted equations does not appear to depend upon $\delta$.  For all $0 < \delta < 1$, the $k^*$ value defined in Lemma \ref{lem:probbound} appears to be large enough to guarantee convergence within the detection horizon.

\begin{figure}
	\includegraphics[width=0.32\textwidth]{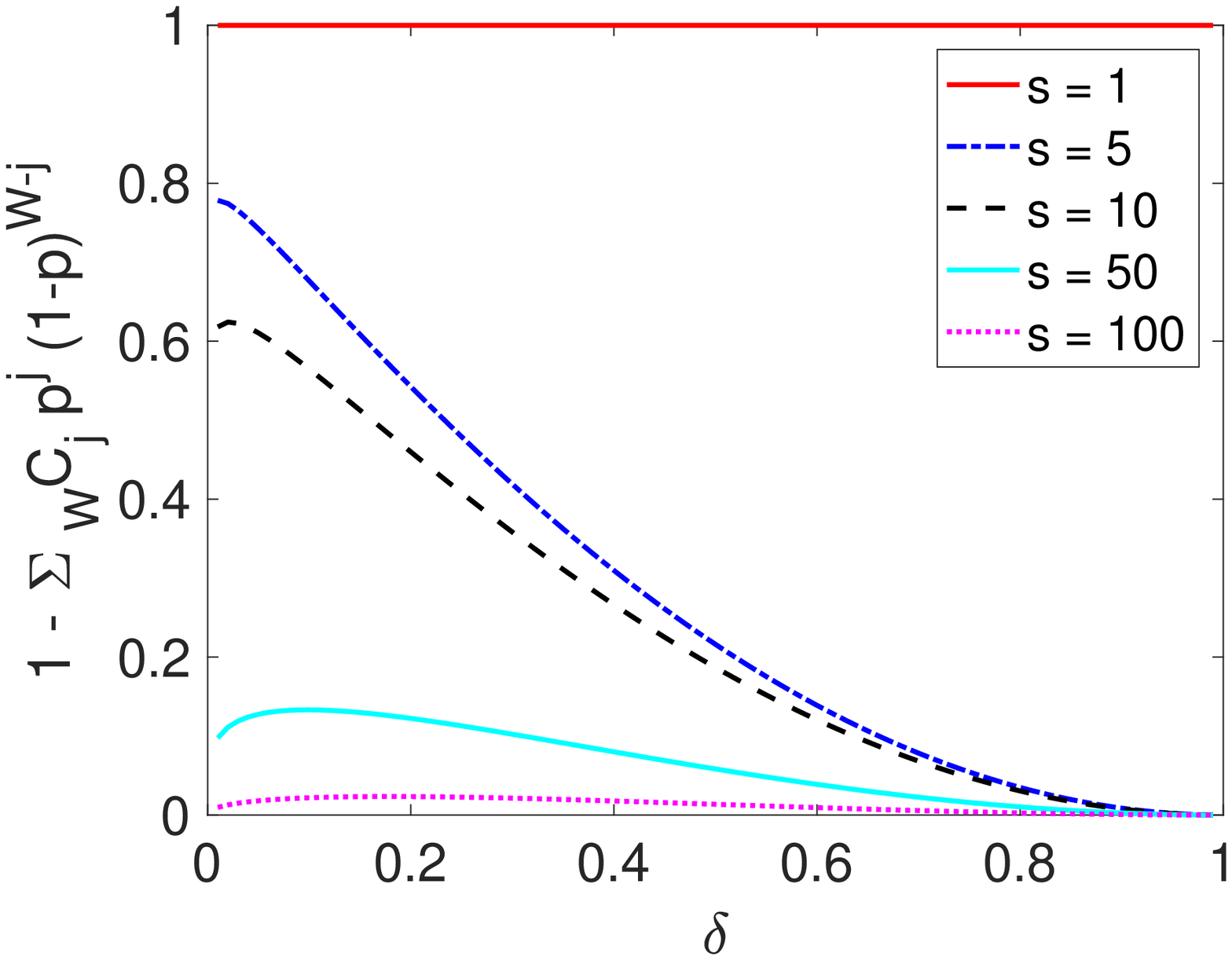}
	\includegraphics[width=0.32\textwidth]{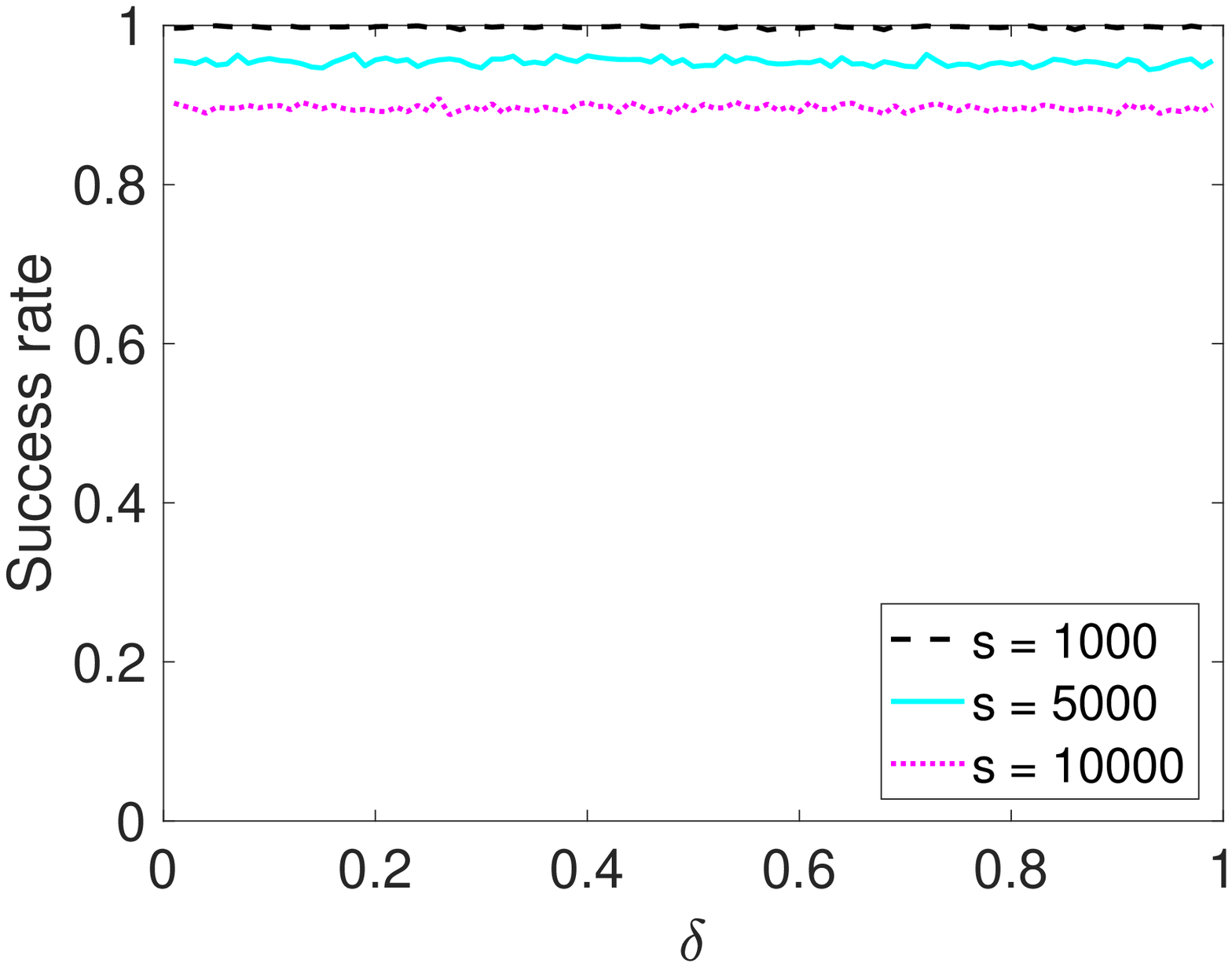}
	\includegraphics[width=0.32\textwidth]{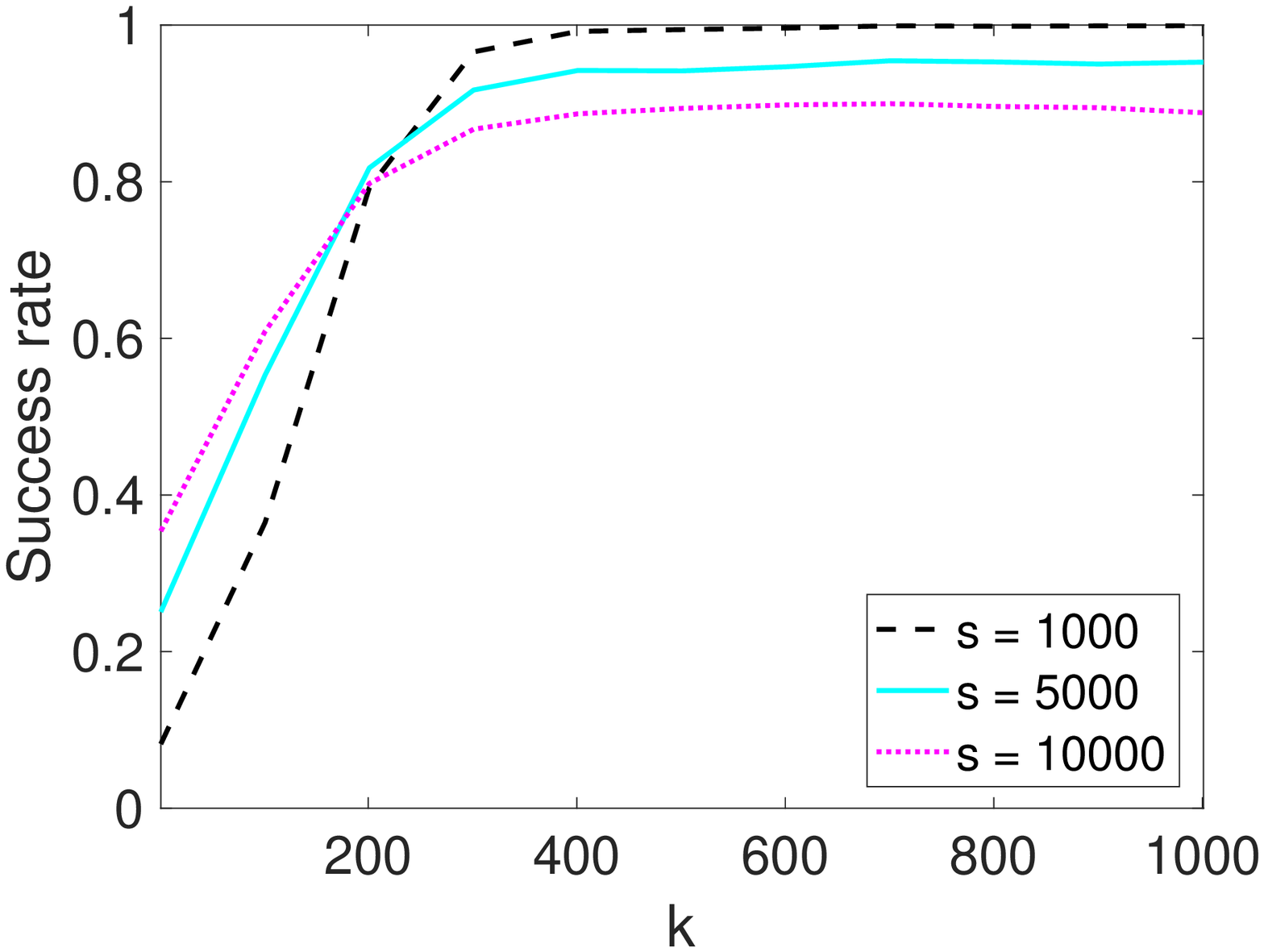}
	\caption{Plots for Method \ref{WKwoRUS} on $50000 \times 100$ Gaussian system (normalized) with various number of corrupted equations, $s$.  Left: Bound given in Theorem \ref{thm:dsmallerthanIwindowedprob} for Method \ref{WKwoRUS} with $W = 2$, $d = \lceil s/2 \rceil$ and $k^*$ as given in Lemma \ref{lem:probbound}; middle: Average fraction of corrupted equations detected after $W = \lceil s/d \rceil$ rounds recording $d = \lceil s/10 \rceil$ equations per round with $k^*$ (as given in Lemma \ref{lem:probbound}) RK iterations per round for varying $\delta$ in 100 trials; right: Average fraction of corrupted equations detected after $W = \lceil s/d \rceil$ rounds recording $d = \lceil s/10 \rceil$ equations per round with varying $k$ (number of RK iterations per round) in 100 trials.}\label{fig:WKwoRUS}
\end{figure}

In Figure \ref{fig:WKwoRUS}, we briefly explore the theoretical guarantees for Method \ref{WKwoRUS} given in Theorem \ref{thm:dsmallerthanIwindowedprob}, and the empirical behavior of this method.  These plots are for a $50000 \times 100$ Gaussian system (normalized) with various number, $s$, of corrupted equations.  We randomly sample $s$ entries of the right hand side vector $b$ and corrupt them by adding $1$, so $\epsilon^* = 1$.  The plot on the left of Figure \ref{fig:WKwoRUS} plots the lower bound on the probability of selecting all corrupted equations given in Theorem \ref{thm:dsmallerthanIwindowedprob} for $W = 2$, $d = \lceil s/2 \rceil$, and $k^*$ as given in Lemma \ref{lem:probbound}.  Meanwhile in the middle and right plots of Figure \ref{fig:WKwoRUS}, we plot the average fraction of corrupted equations recorded for Method \ref{WKwoRUS} with $W = \lceil s/d \rceil$ and $d = \lceil s/10 \rceil$ for 100 trials.  The middle plot has $k^*$ (as given in Lemma \ref{lem:probbound}) RK iterations per round for varying $\delta$, while the right plot has varying $k$ (number of RK iterations per round).  Note that this experiment is different from the others we present in this section as we display the average fraction of corrupted equations recorded over 100 trials, rather than the fraction of trials which detected all corrupted equations.  We note that the theoretical bound plotted on the left of Figure \ref{fig:WKwoRUS} is even more pessimistic than of Method \ref{WKwoR}, but meanwhile the empirical performance of Method \ref{WKwoRUS} plotted in the middle and right of Figure \ref{fig:WKwoRUS} is even better than that seen for Method \ref{WKwoR}.  For this reason, we do not plot these bounds (Theorem \ref{thm:dsmallerthanIwindowedprob}) or the performance of Method \ref{WKwoRUS} for additional systems as we expect the results to trend similarly for other systems.

The figures for Method \ref{WKwoR} mentioned above are recreated for a system whose rows are more correlated ($A \in \R^{50000 \times 100}$ with $a_{ij} \sim \mathcal{N}(1,0.5)$ then normalized) in Figures \ref{fig:correlated} and \ref{fig:correlatedtotal}.  The system is corrupted in randomly selected right-hand side entries with random integers in $[1,5]$ so that $\epsilon^* = 1$.  For these plots and experiments, $d=s$.  The upper left image of Figure \ref{fig:correlated} plots the $k^*$ values defined in Lemma \ref{lem:probbound} for this system, and the upper middle image plots the theoretically guaranteed probability of selecting all $s$ corrupted equations in a single round of Method \ref{WKwoR}.  The upper right image of Figure \ref{fig:correlated} plots the theoretically guaranteed probability of selecting all $s$ corrupted equations in one round out of the $W = \lfloor \frac{m-n}{s}\rfloor$, while the lower left image plots the ratio of successful trials, in which all $s$ corrupted equations were selected in one round of the $W$, out of 100 trials.  The lower right plot of Figure \ref{fig:correlated} plots how this ratio changes as the number of RK iterations, $k$, in each round varies.  Finally, Figure \ref{fig:correlatedtotal} plots the ratio of successful trials, in which all $s$ corrupted equations were selected after all $W$ rounds, out of 100 trials as one varies $\delta$ (left) and $k$ (right).  We note that the discrepancy between the lower bound on the probability of successfully detecting all corrupted equations in one round (upper right plot) and the experimental rate of detecting all corruption (lower left plot) is even larger than in the case of Gaussian systems.  

\begin{figure}
	\includegraphics[width=0.33\textwidth]{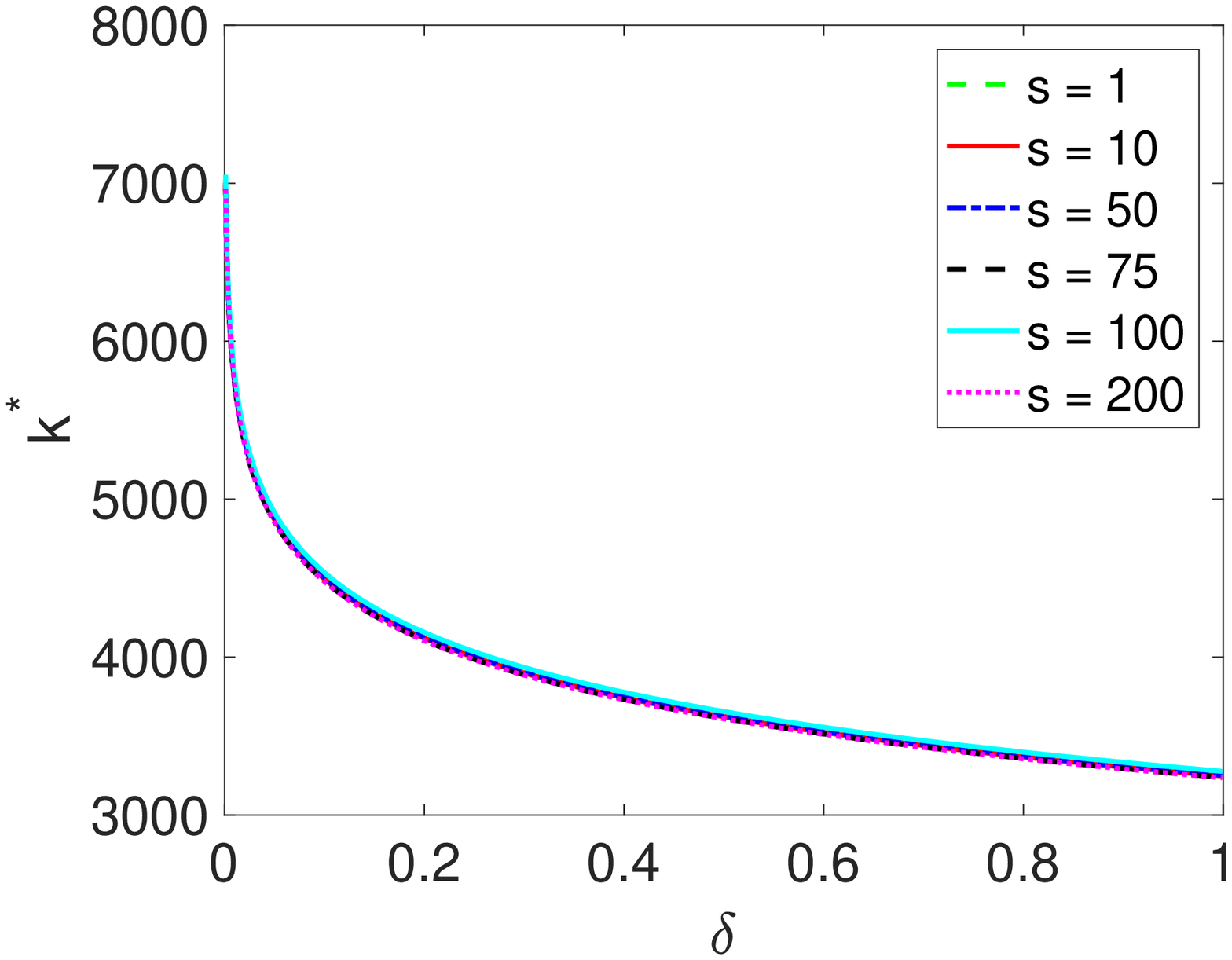} \includegraphics[width=0.33\textwidth]{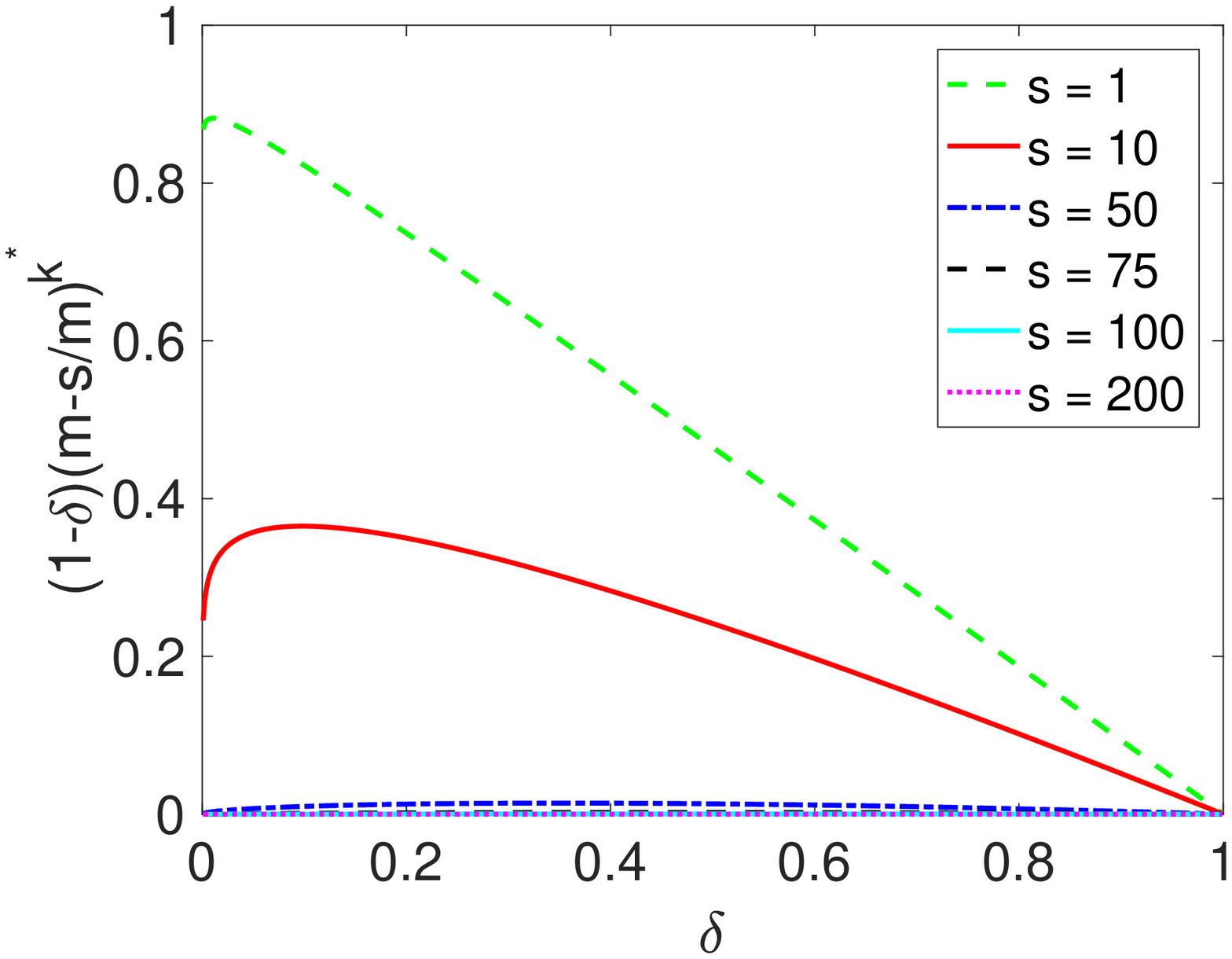}
	\includegraphics[width=0.33\textwidth]{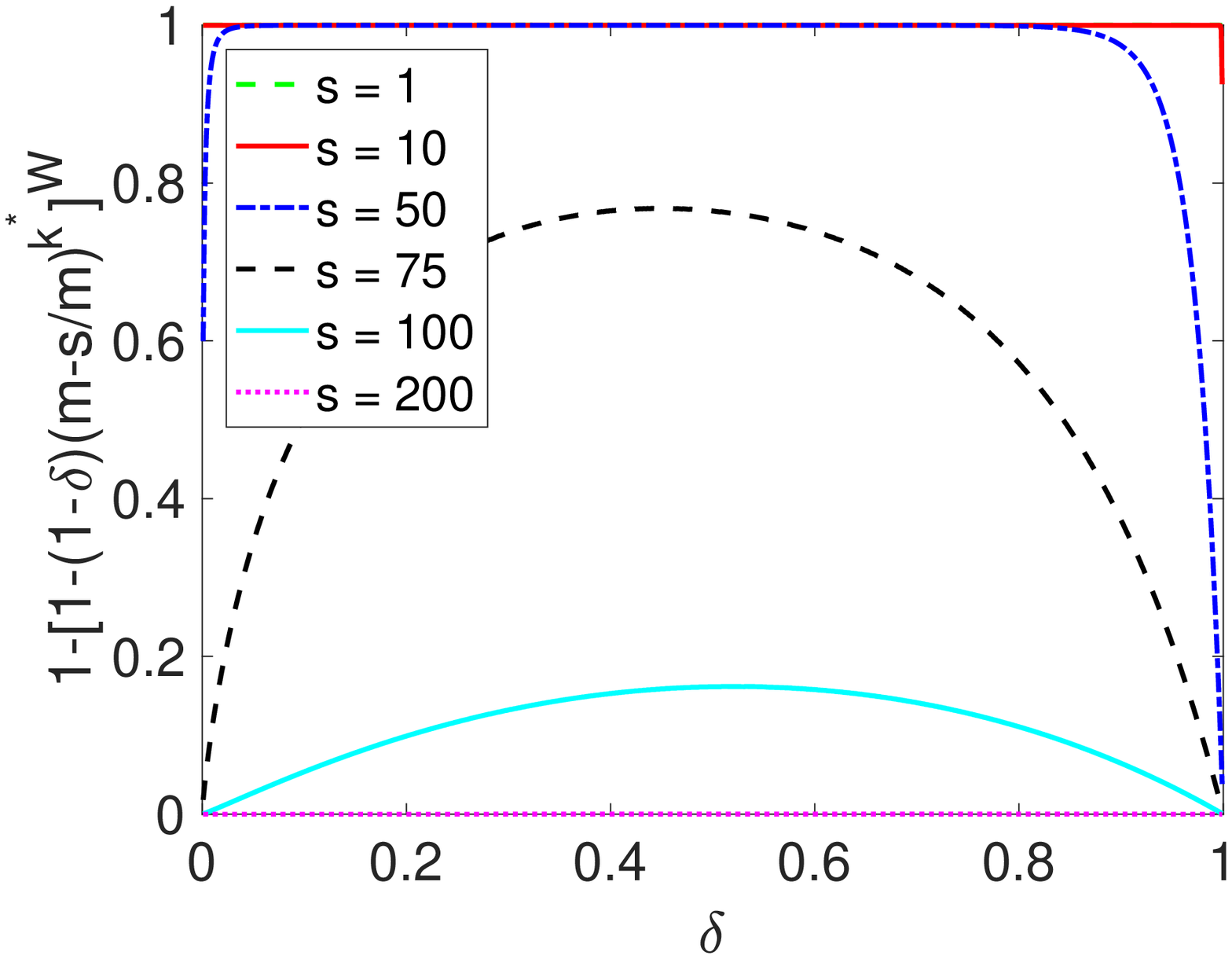} \\\includegraphics[width=0.49\textwidth]{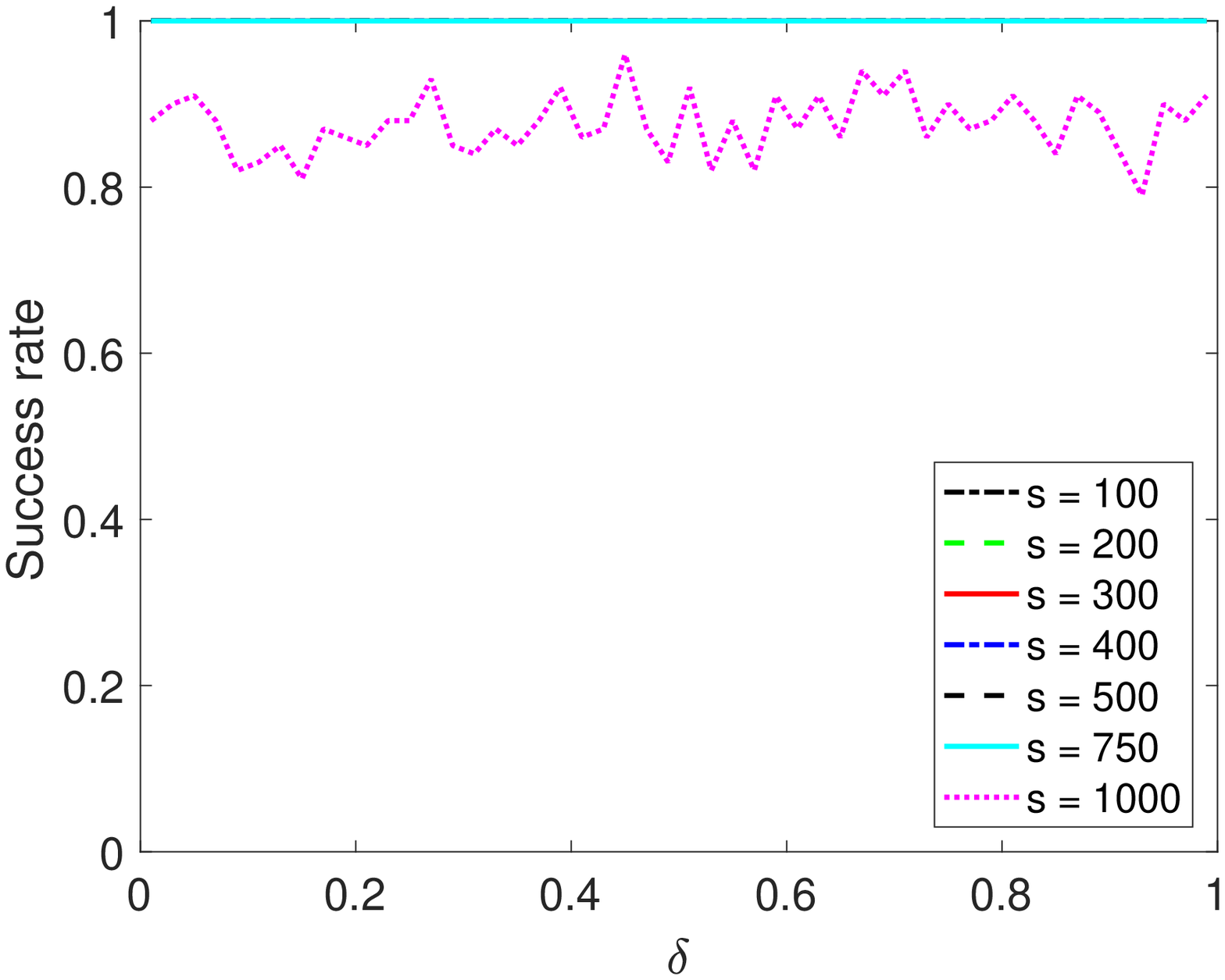}
	\includegraphics[width=0.49\textwidth]{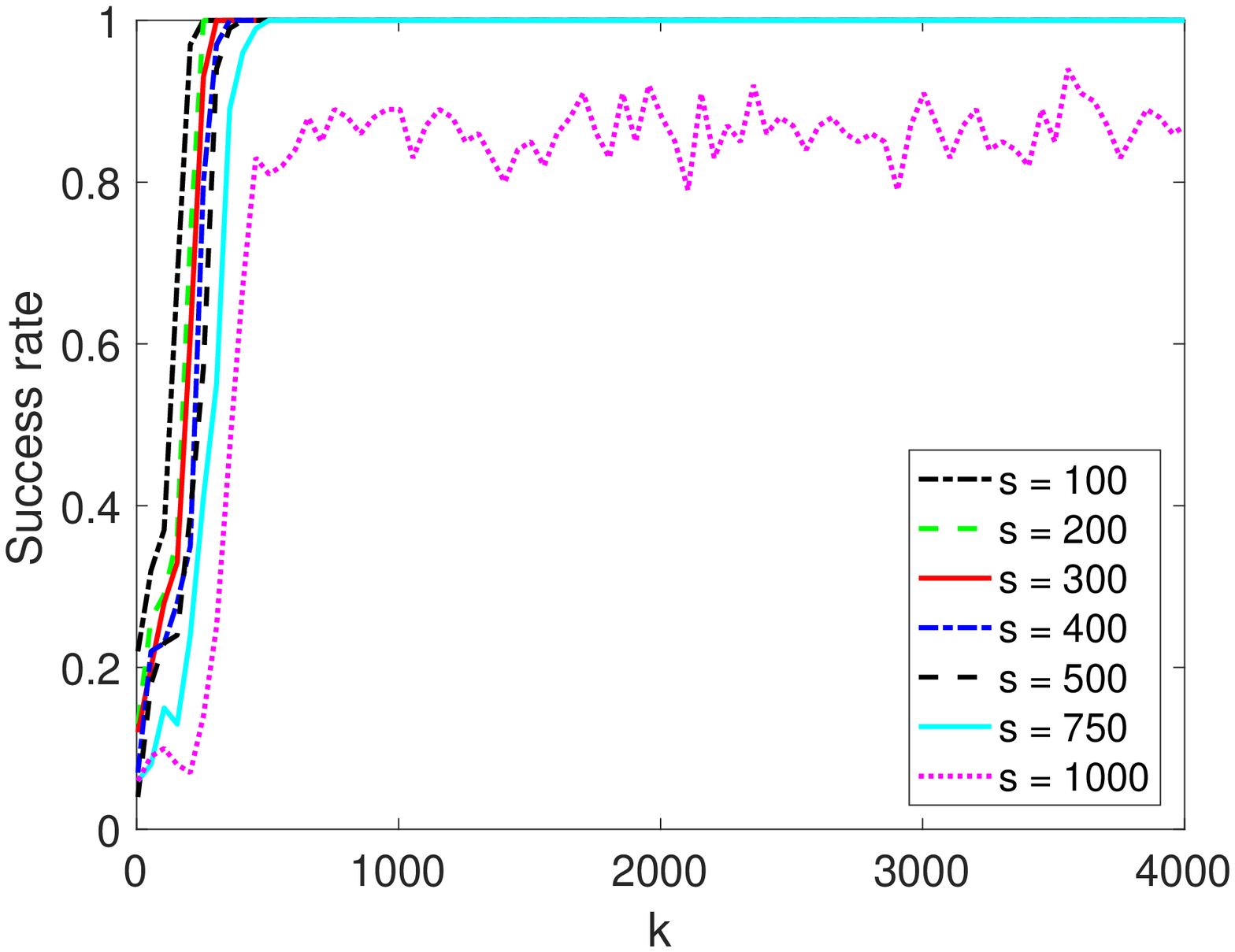}
	\caption{Plots for Method \ref{WKwoR} on $50000 \times 100$ correlated system (normalized) with various number of corrupted equations, $s$.  Upper left: $k^*$ as given in Lemma \ref{lem:probbound}; upper middle: Lower bound on probability of successfully detecting all corrupted equations in single round as given in Lemma \ref{lem:probbound}; upper right: Lower bound on probability of successfully detecting all corrupted equations in one round out of $W = \lfloor \frac{m-n}{s} \rfloor$ rounds as given in Theorem \ref{thm:d=Iwindowedprob};  lower left: Experimental ratio of success of detecting all $s$ corrupted equations in one round out of $W = \lfloor \frac{m-n}{s} \rfloor$ for choice of $\delta$; lower right: Experimental ratio of success of detecting all $s$ corrupted equations in one round out of $W = \lfloor \frac{m-n}{s} \rfloor$ for choice of $k$ (number of RK iterations per round).}
	\label{fig:correlated}
\end{figure}

\begin{figure}
	\begin{center}
		\includegraphics[width=0.49\textwidth]{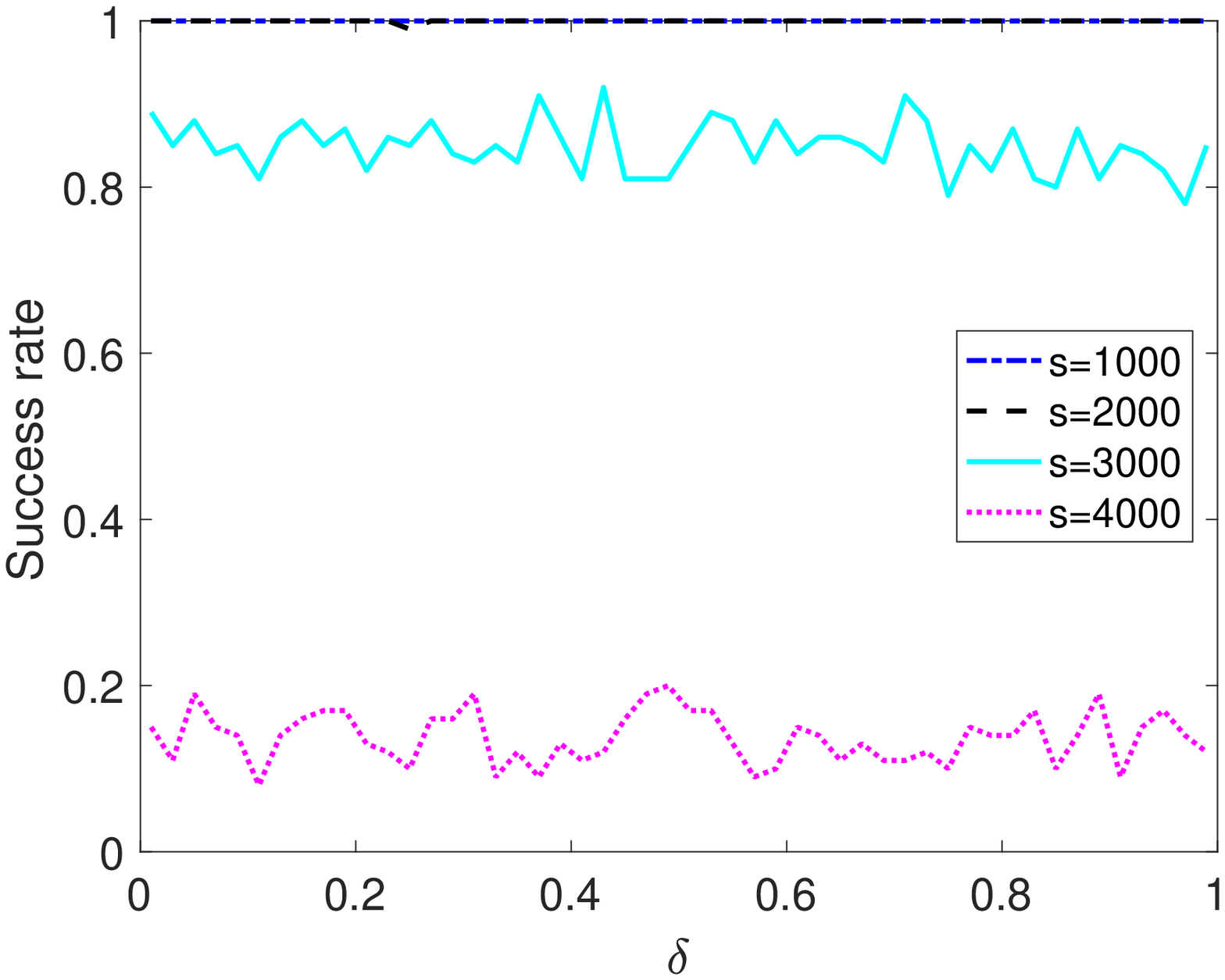}
		\includegraphics[width=0.49\textwidth]{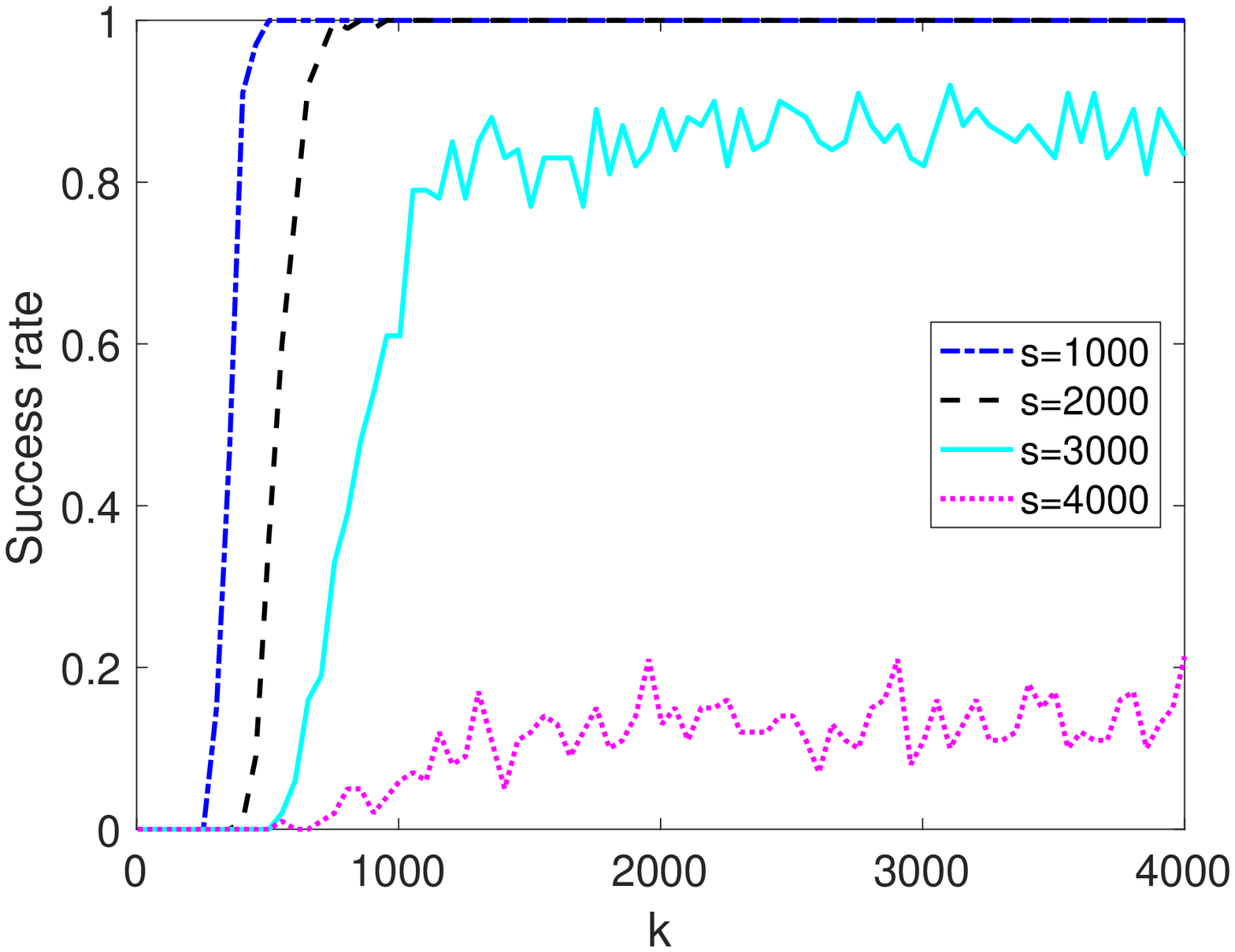}
	\end{center}
	\caption{Plots for Method \ref{WKwoR} on $50000 \times 100$ correlated system (normalized) with various number of corrupted equations, $s$. Left: Experimental ratio of successfully detecting all $s$ corrupted equations after all $W = \lfloor \frac{m-n}{s} \rfloor$ rounds for choice of $\delta$; right: Experimental ratio of successfully detecting all $s$ corrupted equations after all $W = \lfloor \frac{m-n}{s} \rfloor$ rounds for choice of $k$ (number of RK iterations per round).}
	\label{fig:correlatedtotal}
\end{figure}

First, note that $k^*$, as given in Lemma \ref{lem:probbound}, depends very weakly upon $s$.  In the upper left plots of Figures \ref{fig:gaussian} and \ref{fig:correlated}, the values of $k^*$ plotted are very slightly different for different values of $s$ (the line thickness makes these distinct lines appear as one).  Note that the definition of $k^*$ (the theoretically required number of RK iterations to reach the detection horizon) is defined by the theoretical convergence rate which can be quite pessimistic.  As has been seen in the lower right plots of Figures \ref{fig:gaussian} and \ref{fig:correlated}, and in the right plots of Figures \ref{fig:gaussiantotal} and \ref{fig:correlatedtotal}, detection can be successful with a significantly smaller choice of $k$.  Note that in Figure \ref{fig:gaussian}, the theoretically required $k^*$ value is between $600$ and $1400$ but $k > 500$ seems to perform well.  Likewise, in Figure \ref{fig:correlated}, the theoretically required $k^*$ value is between $3000$ and $8000$ but $k > 500$ seems to perform well.  It is unsurprising that this bound is even more pessimistic for the correlated system, as the conditioning of a correlated system causes the RK convergence guarantee to be quite poor, while experimentally we see a much faster rate of convergence.

\subsection{Implementation considerations}There are several options for $d$, some more practically feasible than others.  Our theoretical results are probabilistic guarantees for Method \ref{WKwoR} with $d \ge s$, which of course cannot be known in practice, as well as for Method \ref{WKwoRUS} with  $d \ge 1$, which is practical, but the method is more expensive computationally. In practice, one could choose $d$ as the user estimate for $s$.

The choice of $d$ and $W$ are complementary in that increasing $d$ decreases $W$ (since one may have less rounds if in each round more equations are selected).  In selecting $d$ and $W$, we wish to balance the desire to increase $d$ in order to record all of the corrupted equations when we have a successful round with the fact that as $d$ grows, we can have less rounds.  We never discard or record more than $m-n$ of the constraints, as at the end of any method, we wish to have a full rank linear system of equations remaining whose solution is $\ve{x}^*$, the pseudo-solution.  Thus, for any $d$ we may not run more than $W = \lfloor\frac{m-n}{d}\rfloor$ rounds.  However, in practice, this choice of $W$ may be larger than is necessary.  This is explored in Figures \ref{fig:dtotalprob} and \ref{fig:Wtotalprob}.  In the experiment producing Figure \ref{fig:dtotalprob}, we ran $W = \lfloor \frac{m-n}{d} \rfloor$ rounds of Method \ref{WKwoR} with $k^*$ (defined in Lemma \ref{lem:probbound}) RK iterations selecting $d$ equations each round, and record the ratio of successful trials, which selected all $s$ corrupted equations after all rounds, out of 100 trials.  The figure on the left shows the results for a Gaussian system, while the figure on the right shows the results for a correlated system.  In the experiment producing Figure \ref{fig:Wtotalprob}, we ran $W \le \lfloor \frac{m-n}{s} \rfloor$ rounds of Method \ref{WKwoR} with $k^*$ (defined in Lemma \ref{lem:probbound}) RK iterations selecting $s$ equations each round, and record the ratio of successful trials, which selected all $s$ corrupted equations after all rounds, out of 100 trials.  The figure on the left is for a Gaussian system, while the figure on the right is for a correlated system.  Both are corrupted with random integers in $[1,5]$ in randomly selected entries of $\ve{b}$, so $\epsilon^* = 1$.  

\begin{figure}
	\begin{center}
		\includegraphics[width=0.49\textwidth]{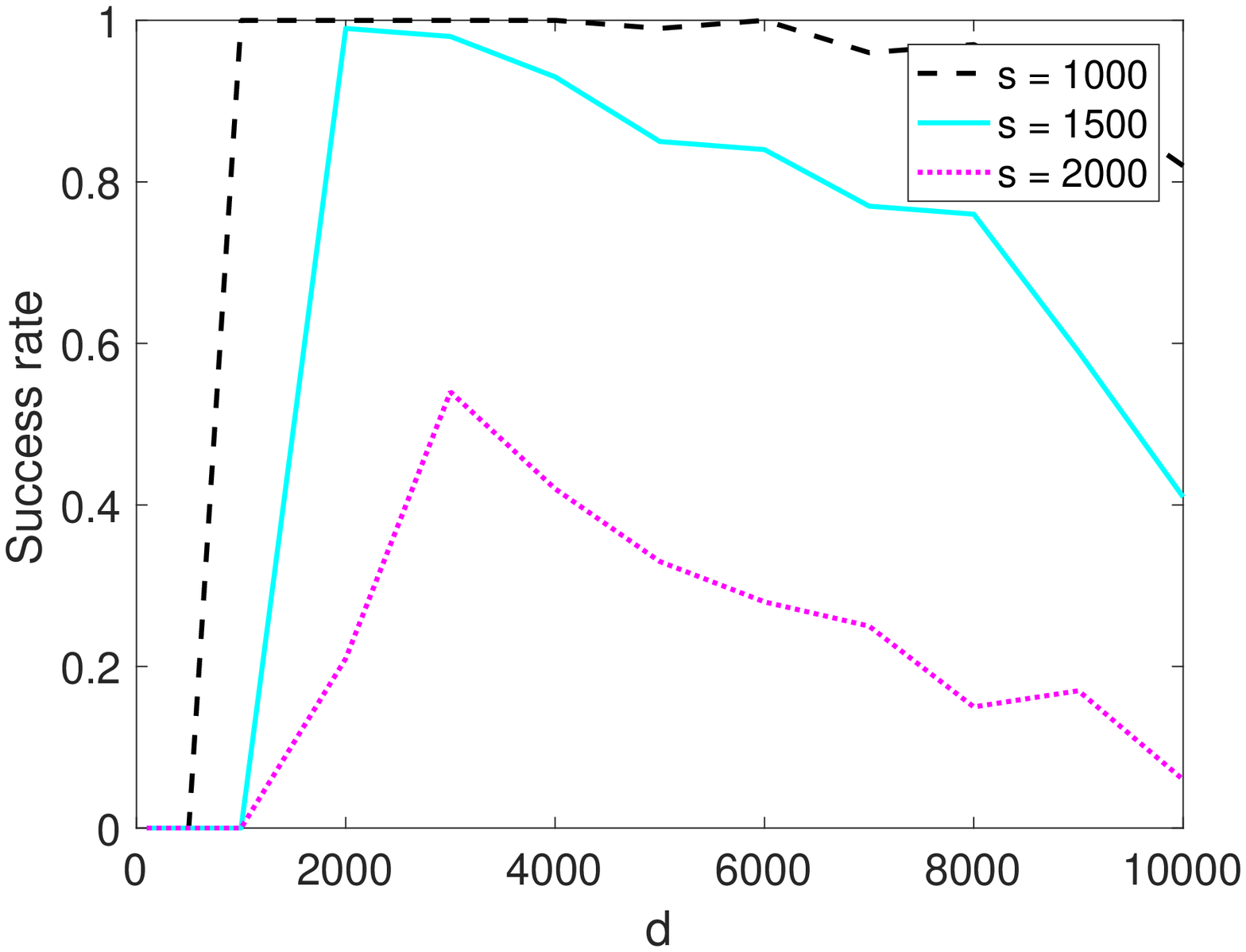}
		\includegraphics[width=0.49\textwidth]{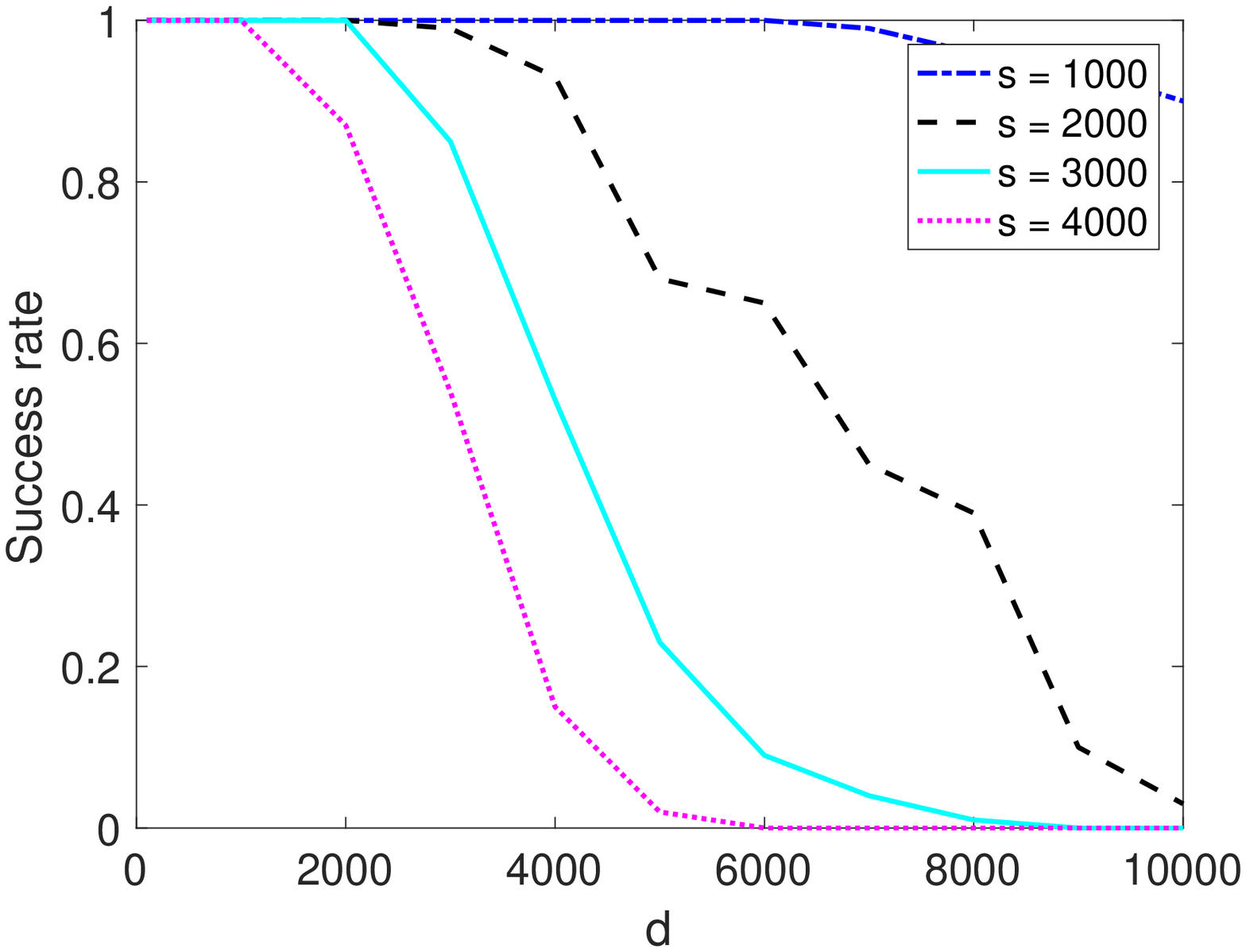}
	\end{center}
	\caption{Left: Experimental ratio of success of detecting all $s$ corrupted equations after all $W =  \lfloor\frac{m-n}{d}\rfloor$ rounds of $k^*$ (as given in Lemma 2) RK iterations selecting $d$ equations for $50000 \times 100$ Gaussian system with $s$ corrupted equations. Right: Experimental ratio of success of detecting all $s$ corrupted equations after all $W = \lfloor\frac{m-n}{d}\rfloor$ rounds of $k^*$ (as given in Lemma 2) RK iterations selecting $d$ equations for $50000 \times 100$ correlated system with $s$ corrupted equations.}
	\label{fig:dtotalprob}
\end{figure}

\begin{figure}
	\begin{center}
		\includegraphics[width=0.49\textwidth]{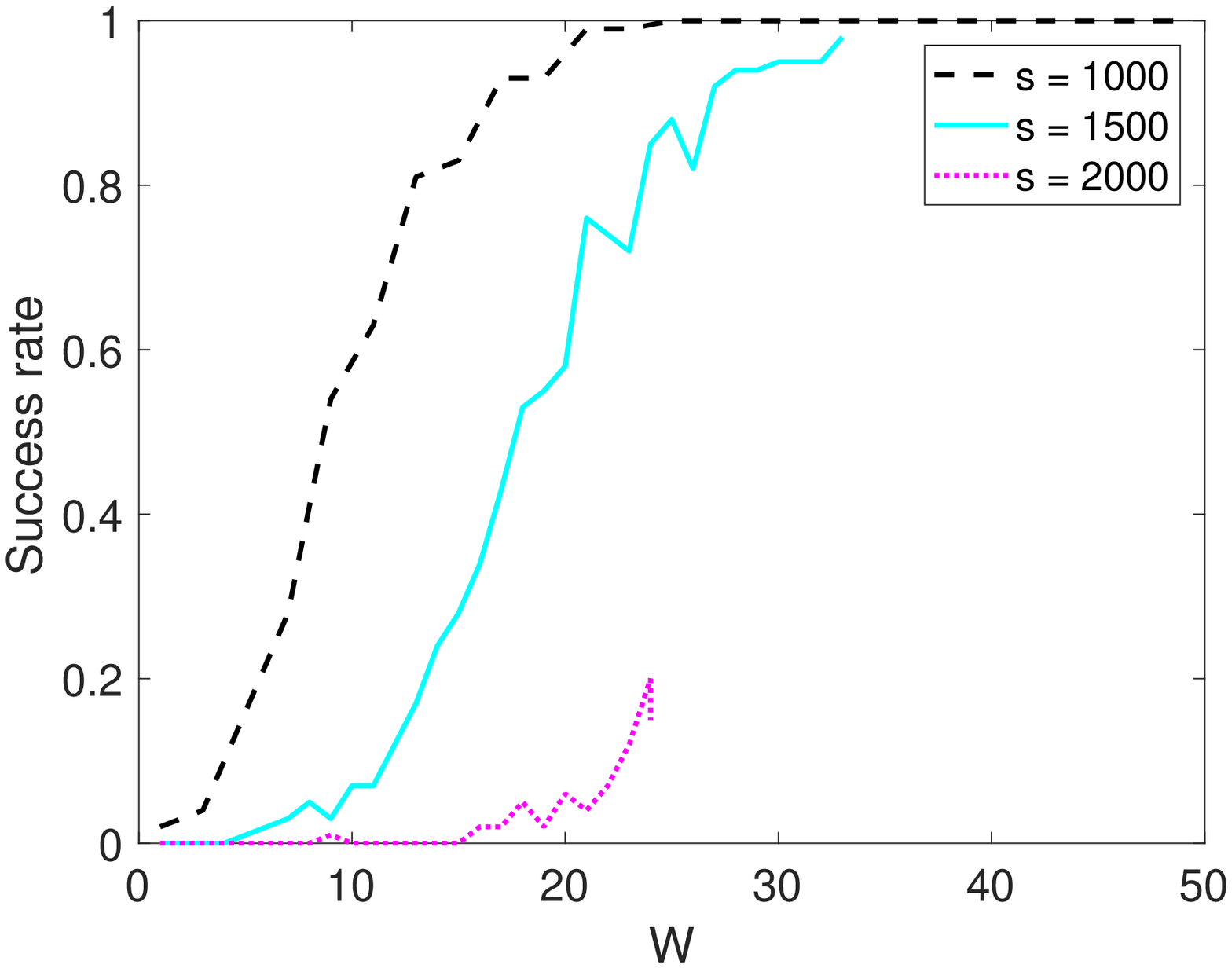}
		\includegraphics[width=0.49\textwidth]{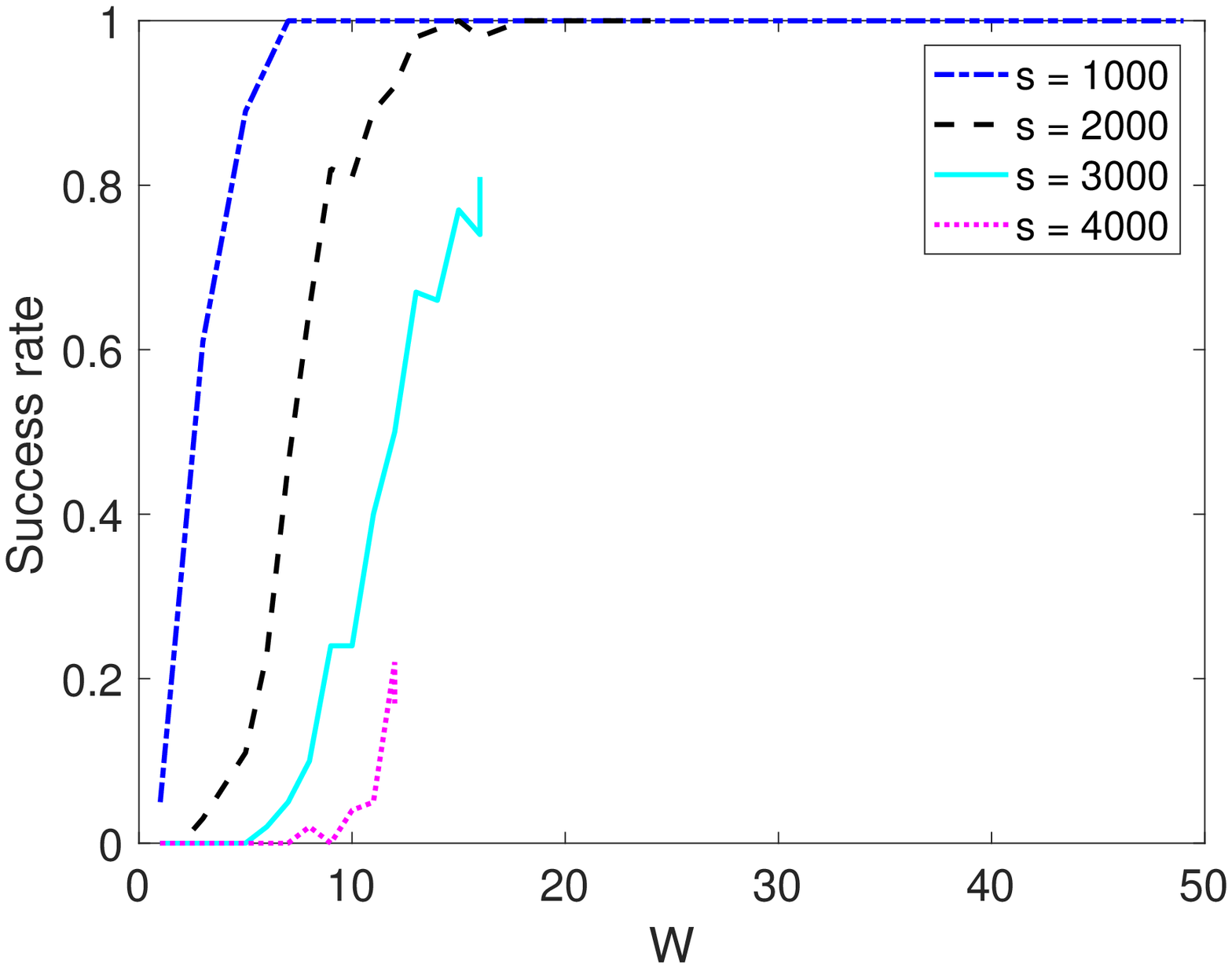}
	\end{center}
	\caption{Left: Experimental ratio of success of detecting all $s$ corrupted equations after all $W \le \lfloor \frac{m-n}{s} \rfloor$ rounds of $k^*$ (as given in Lemma 2) RK iterations for $50000 \times 100$ Gaussian system with $s$ corrupted equations. Right: Experimental ratio of success of detecting all $s$ corrupted equations after all $W \le \lfloor \frac{m-n}{s} \rfloor$ rounds of $k^*$ (as given in Lemma 2) RK iterations for $50000 \times 100$ correlated system with $s$ corrupted equations.}
	\label{fig:Wtotalprob}
\end{figure}

Method \ref{WKwR}, despite not having independent rounds, performs well in practice as is seen in Figure \ref{fig:removalratio}.  In this experiment, we perform $W = \lfloor \frac{m-n}{s} \rfloor$ rounds of Method \ref{WKwR} with $k$ RK iterations, removing $s$ equations each round.  The plot shows the ratio of successful trials, in which all $s$ corrupted equations are removed after all $W$ rounds, out of 100 trials.  The method is run on a Gaussian system which is corrupted by random integers in $[1,5]$ in randomly selected entries of $\ve{b}$, so $\epsilon^* = 1$.

\begin{figure}
	\includegraphics[width=0.6\textwidth]{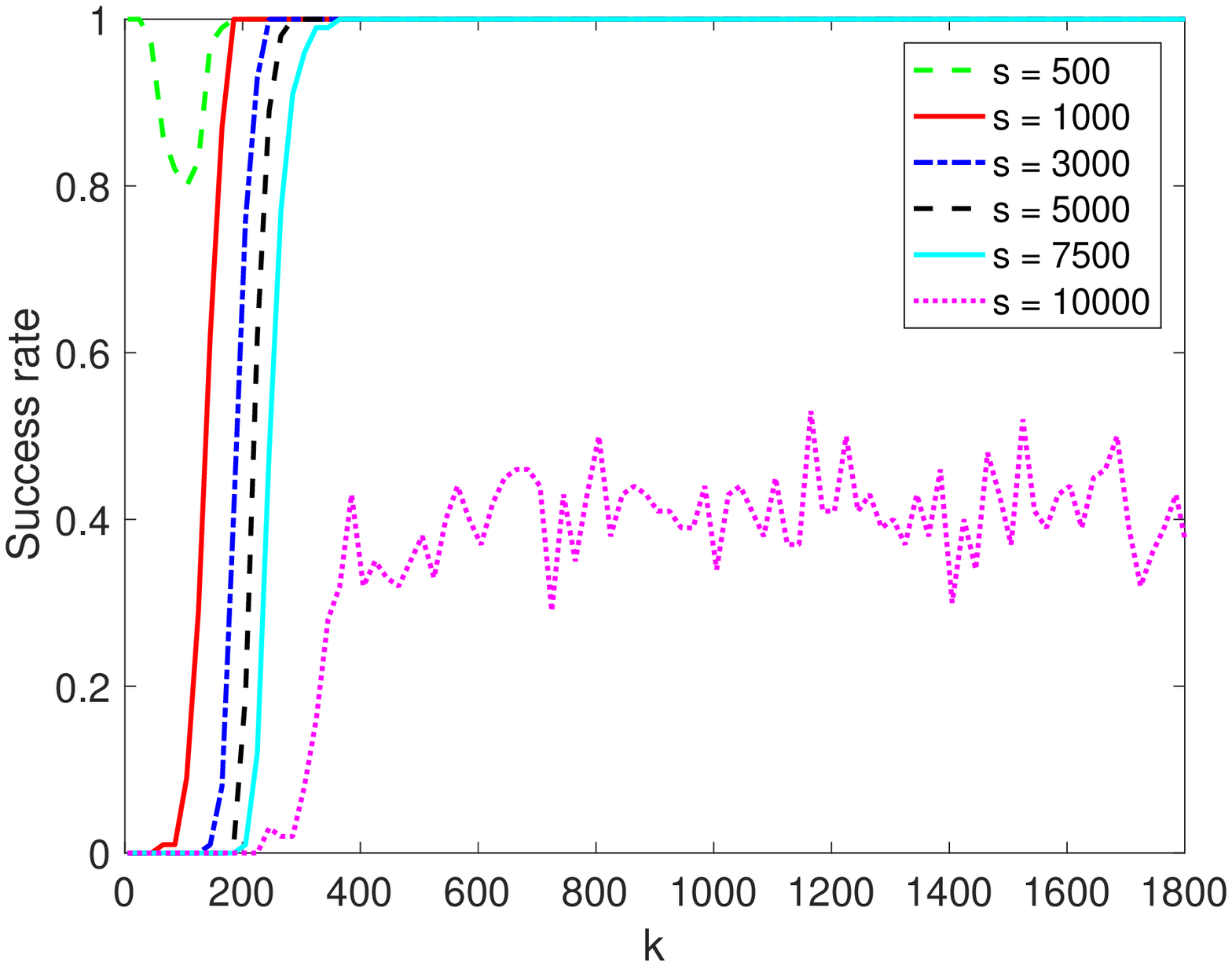}
	\caption{Experimental ratio of success of removing (Method \ref{WKwR}) all $s$ corrupted equations after all $W = \lfloor \frac{m-n}{s} \rfloor$ rounds for $50000\times 100$ Gaussian system with $s$ corrupted equations and choice of $k$.}\label{fig:removalratio}
\end{figure}

\subsection{Real Data Experiments}\label{sec:realdataexperiments}

We additionally test the methods on real data.  Our first experiments are on tomography problems, generated using the Matlab
Regularization Toolbox by P.C. Hansen (\url{http://www.imm.dtu.dk/~pcha/Regutools/})
\cite{hansen2007regularization}. In particular we present a 2D tomography problem $A\ve{x} = \ve{b}$ for an $m \times n$ matrix with $m = f N^2$ and $n = N^2$. Here $A$ corresponds to the absorption along a random line through an $N \times N$ grid. In our experiments we set $N = 20$ and the oversampling factor $f = 3$.  This yielded a matrix $A$ with condition number $\kappa(A) = 2.08$.  As the resulting system was consistent, we randomly sampled $s = 100$ constraints uniformly from among the rows of $A$ and corrupted the right-hand side vector $\ve{b}$ by adding $1$ in these entries, so $\epsilon^* = 1$.  This corrupted system has $k^* = 66334$ (as given in Lemma \ref{lem:probbound}).  Figure \ref{fig:tomorates} contains the average fraction of the $s=100$ corrupted equations detected or removed for Methods \ref{WKwoR} (left) and \ref{WKwR} (right) after all $W = \lfloor\frac{m - n}{d}\rfloor$ rounds for various values of $k$ (RK iterations per round) for 100 trials.

\begin{figure}
	\includegraphics[width=0.49\textwidth]{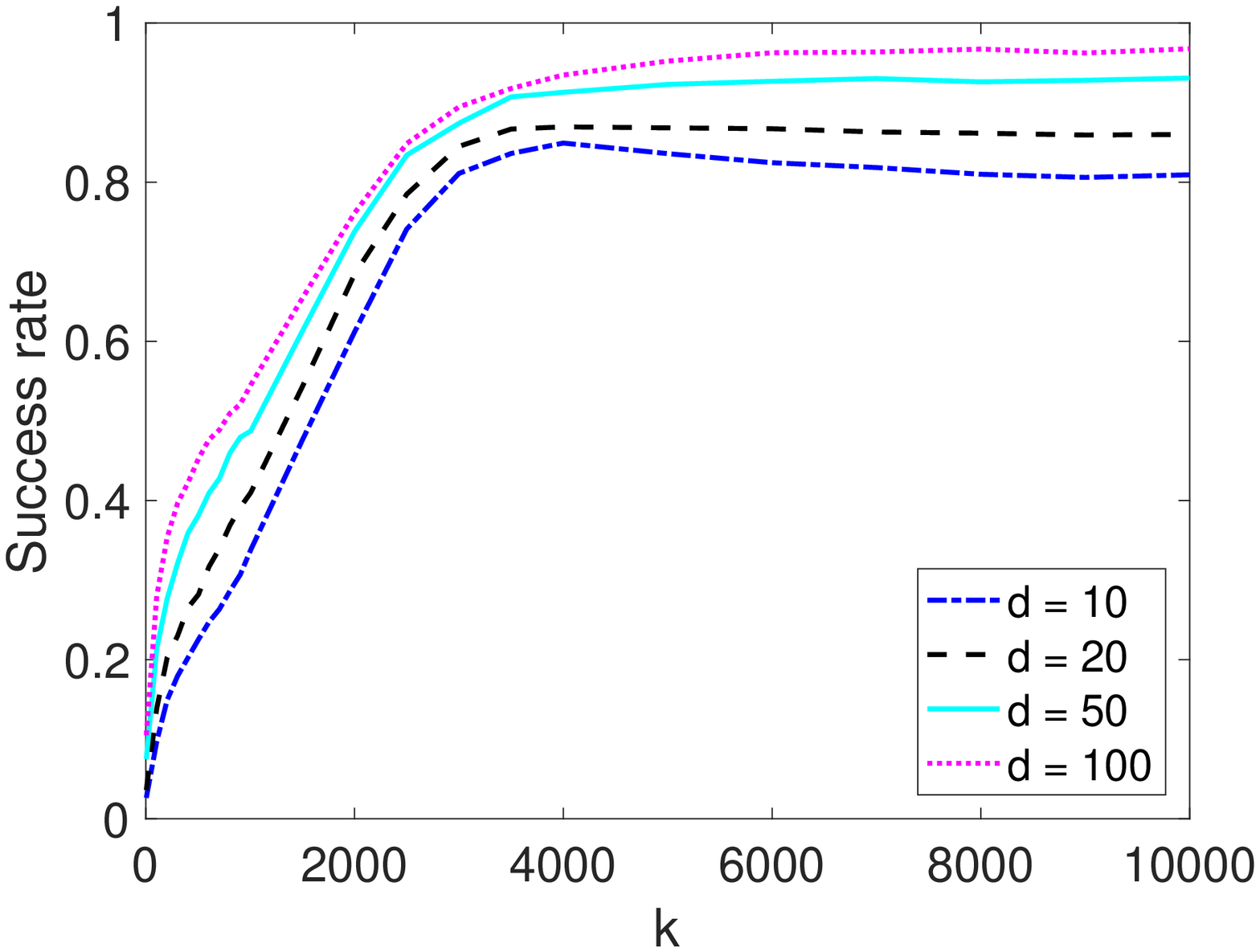}
	\includegraphics[width=0.49\textwidth]{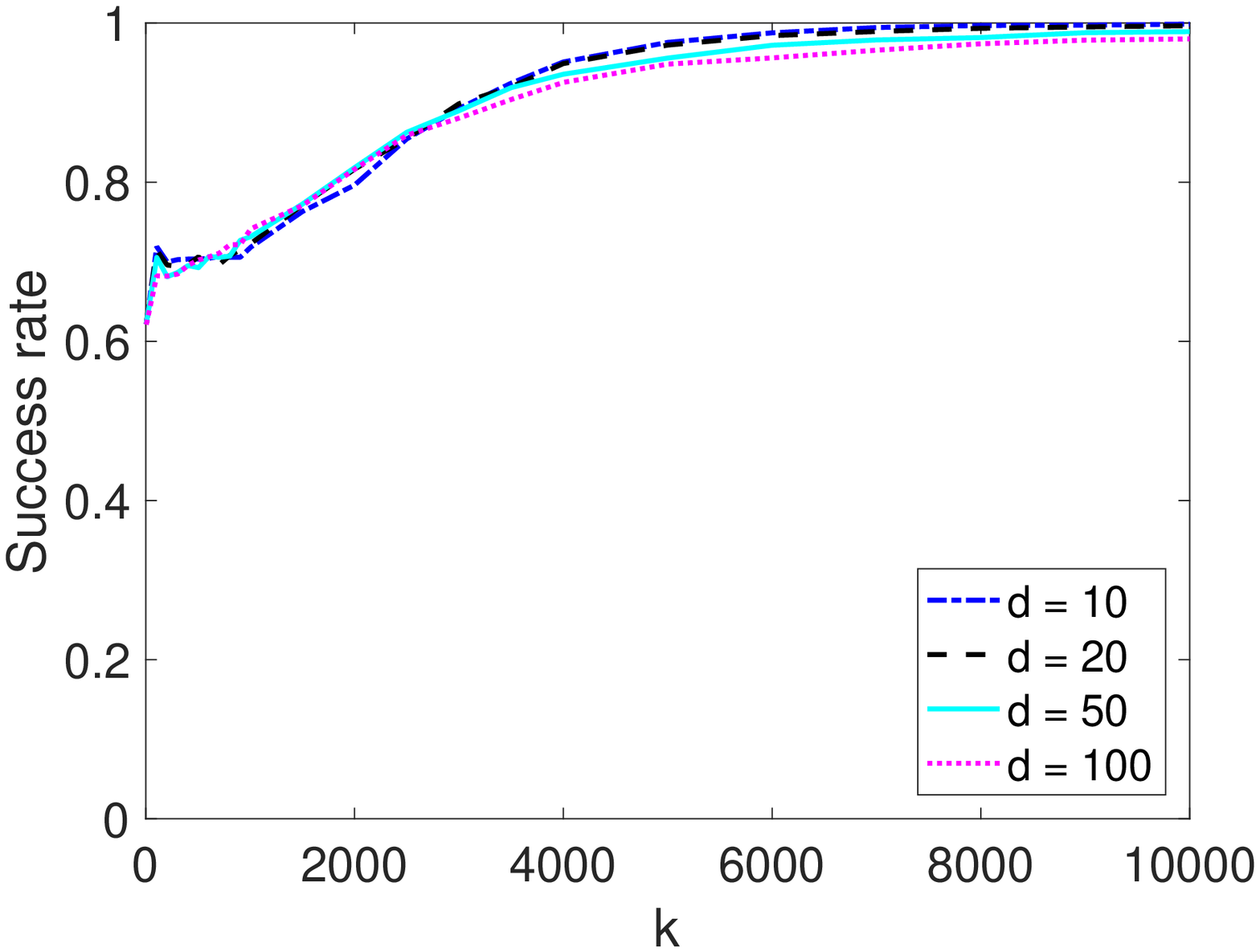}
	\caption{Plots for $1200 \times 400$ tomography system with $s = 100$ corrupted equations.  Left: average fraction of corrupted equations detected in 100 trials after all $W = \lfloor\frac{m - n}{d}\rfloor$ rounds of Method \ref{WKwoR}; right: average fraction of corrupted equations removed in 100 trials after all $W = \lfloor\frac{m - n}{d}\rfloor$ rounds of Method \ref{WKwR}.}\label{fig:tomorates}
\end{figure}

We also generated corrupted data sets using the Wisconsin (Diagnostic) Breast Cancer data set,
which includes data points whose features are computed from a digitized image of
a fine needle aspirate (FNA) of a breast mass and describe characteristics of the cell nuclei present in the
image \cite{UCI}.  This collection of data points forms our matrix $A \in \mathbb{R}^{699 \times 10}$, we construct $\ve{b}$ to form a consistent system, and then corrupt a random selection of $100$ entries of the right-hand side by adding $1$, so $\epsilon^* = 1$.  This corrupted system has $k^* = 3432$ (as given in Lemma \ref{lem:probbound}).  Figure \ref{fig:breastcancer} contains the average fraction of the $s=100$ corrupted equations detected or removed for Methods \ref{WKwoR} (left) and \ref{WKwR} (right) after all $W = \lfloor \frac{m - n}{d} \rfloor$ rounds for various values of $k$ (RK iterations per round) for 100 trials.

\begin{figure}
	\includegraphics[width=0.49\textwidth]{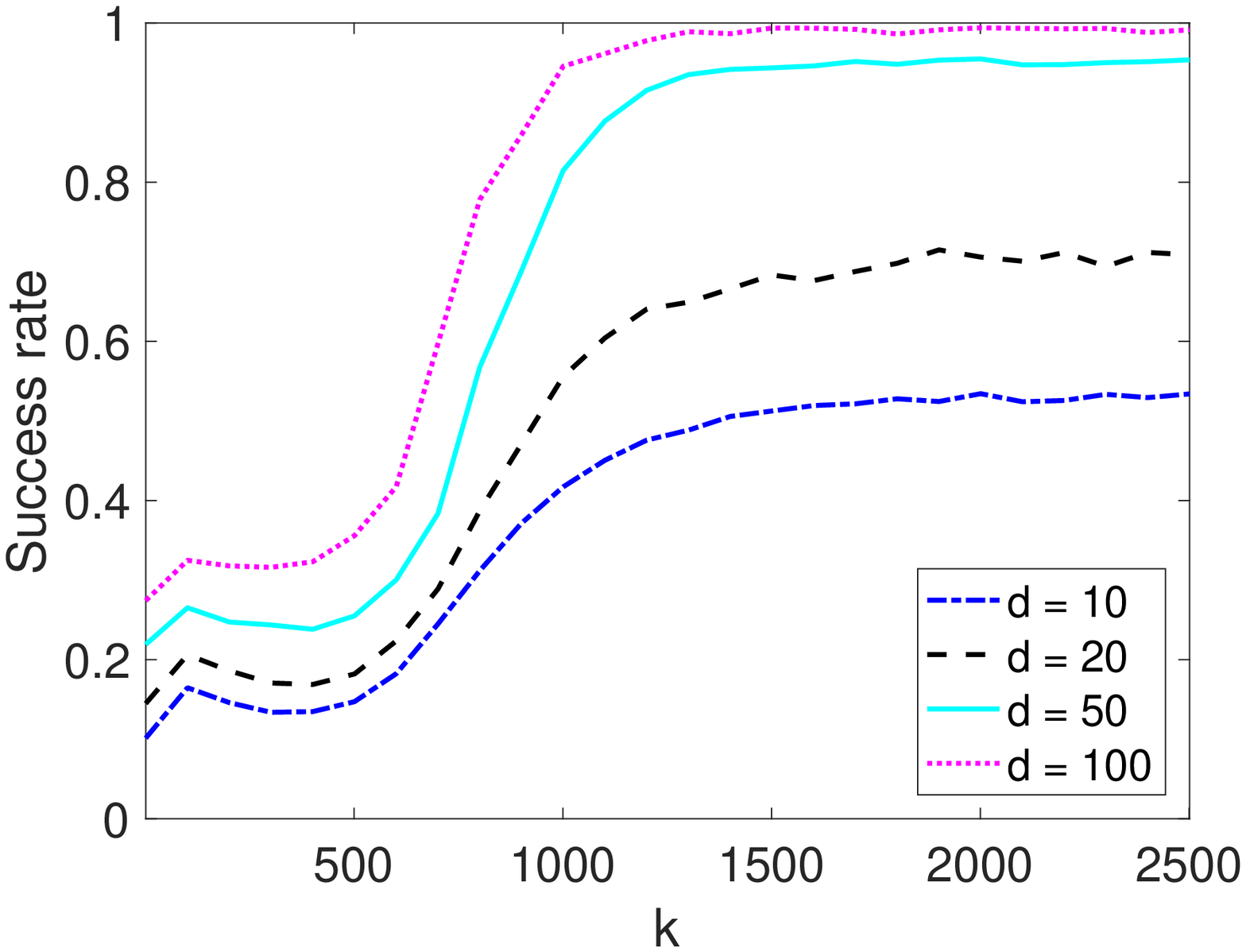}
	\includegraphics[width=0.49\textwidth]{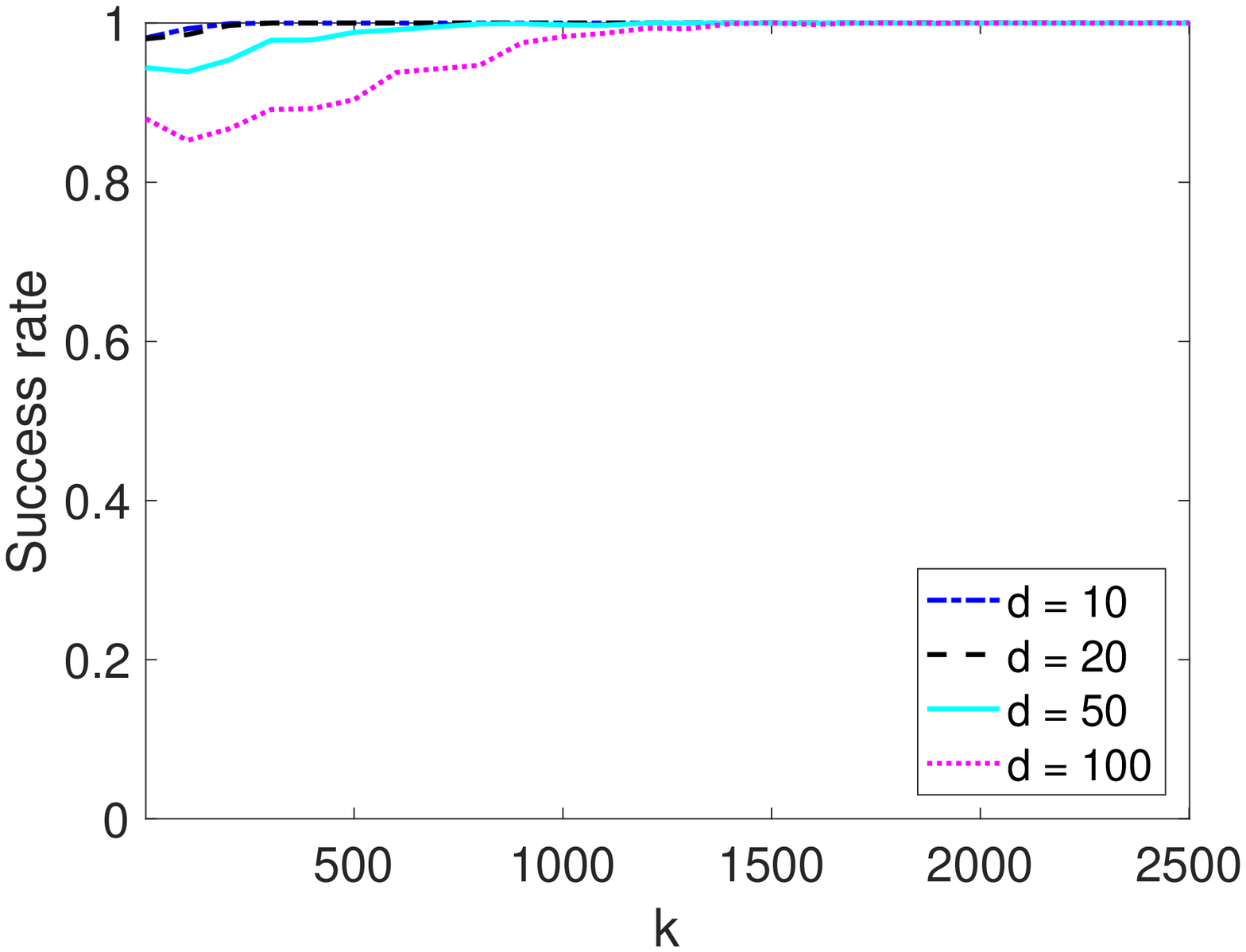}
	\caption{Plots for $699 \times 10$ system defined by Wisconsin (Diagnostic) Breast Cancer data set with $s = 100$ corrupted equations.  Left: average fraction of corrupted equations detected in 100 trials after all $W = \lfloor\frac{m - n}{d}\rfloor$ rounds of Method \ref{WKwoR}; right: average fraction of corrupted equations removed in 100 trials after all $W = \lfloor\frac{m - n}{d}\rfloor$ rounds of Method \ref{WKwR}.}\label{fig:breastcancer}
\end{figure} 

\subsection{Comparison to Existing Methods}

As previously mentioned, although related, a comparison of Methods \ref{WKwR}, \ref{WKwoR}, and \ref{WKwoRUS} to methods designed for \texttt{MAX-FEAS} are not natural.  Methods for \texttt{MAX-FEAS} are designed for a much more general, and harder class of problems than our proposed methods.  Methods for \texttt{MAX-FEAS} seek to carefully identify the largest feasible subproblem, while our methods seek to identify and discard potentially corrupted equations.  

However, one may consider our methods as iteratively computing a solution to the current system of equations with a sparse residual, and deleting those entries corresponding to the non-zero entries in the residual.  Thus, we compare the behavior of our method to convex optimization methods on the problem reformulation $\min \|r\|_1 \text{ s.t. } A\ve{x} - \ve{b} = \ve{r}$, as this reformulation should encourage sparsity in the residual.  In these experiments, we run several rounds of $k$ iterations of various convex optimization methods (initialized with the last iterate from the previous round) implemented in the built-in Matlab \texttt{fmincon} function, and remove $d$ equations corresponding to the largest entries of the computed residual, $\ve{r}$.  We measure the CPU time (using Matlab \texttt{cputime}) required to remove all corrupted equations, and compare this to the CPU time required by Method \ref{WKwR} with the same $d$ and $k$ values to remove all corrupted equations.

We test on the same problem data as described in Subsection \ref{sec:realdataexperiments}.  In Table \ref{tab:WKwR}, we report the CPU time (s) required to remove all 100 corrupted equations by Method \ref{WKwR} with $k = 8000$ and $d = 10$, and the method described above for removing $d = 10$ equations with the algorithms `\texttt{interior-point}', `\texttt{active-set}', and `\texttt{sqp}' with $k = 10$ iterations.    

\begin{table}
	\begin{tabular}{|c || c | c | c | c |}
		\hline
		& Method \ref{WKwR} & \texttt{interior-point} & \texttt{active-set} & \texttt{sqp} \\
		\hline
		tomography & 1.72 & 795.66 & 6435.21 & 3513.30 \\
		\hline
		breast cancer & 1.74 & 202.09 & 2643.15 & 507.98 \\
		\hline
	\end{tabular}
\caption{CPU time (s) required to remove all 100 corrupted equations by Method \ref{WKwR} with $k = 8000$ and $d = 10$, and the method described above for removing $d = 10$ equations with the algorithms `\texttt{interior-point}', `\texttt{active-set}', and `\texttt{sqp}' with $k = 10$ iterations.    
}\label{tab:WKwR}
\end{table}

\section{Conclusion}

We have presented a framework of methods for using randomized projection methods to detect and remove corruptions in a system of linear equations.  We provide theoretical bounds on the probability that these methods will successfully detect and remove all corrupted equations.  Moreover, we provide ample experimental evidence that these methods successfully detect corrupted equations and these results far surpass the theoretical guarantees.  

\bibliographystyle{myalpha}
\bibliography{bib}

\appendix
\section{Proofs of Main Results}

We separate our main theoretical results from their proofs so as to minimize distraction from the progression of the results and plots of the probability bounds demonstrated.

\begin{proof}[Proof of Lemma \ref{lem:detectionhorizon}]
	Suppose $\|\ve{x} - \ve{x}^*\| < \frac{1}{2} \epsilon^*$.  Note that for $\|\ve{a}_i\| = 1$, we have $$|r_i| = |A\ve{x} - \ve{b}|_i = |\ve{a}_i^T\ve{x} - b_i| = \frac{|\ve{a}_i^T\ve{x} - b_i|}{\|\ve{a}_i\|} = d(\ve{x},H_i)$$ where $d(\ve{x},H)$ is the Euclidean distance of $\ve{x}$ to the set $H$ and $H_i = \{\ve{x} : \ve{a}_i^T\ve{x} = b_i\}$ is the hyperplane defined by the $i$th equation.  Next, note that $$d(\ve{x}^*, H_i) = |\ve{a}_i^T \ve{x}^* - b_i| = |b_i^* - b_i| = \begin{cases} |\epsilon _i| &  i \in I  \\ 0 & i \not\in I \end{cases}.$$
	
	Now, consider $r_i$ for $i \in I$.  Denoting by $P_H$ the orthogonal projectiong onto $H$, note that 
	\begin{align*}
	|r_i| &= d(\ve{x}, H_i) = \|P_{H_i}(\ve{x}) - \ve{x}\|\\ 
	&= \|P_{H_i}(\ve{x}) - \ve{x}^* - (\ve{x} - \ve{x}^*)\|\\ 
	&\ge |\|P_{H_i}(\ve{x}) - \ve{x}^*\| - \|\ve{x} - \ve{x}^*\||\\
	&\ge d(\ve{x}^*,H_i) - \|\ve{x} - \ve{x}^*\| \\
	&> \frac{1}{2} \epsilon^*,
	\end{align*}
	where the first inequality follows from the triangle inequality and the second from the fact that $\|P_{H_i}(\ve{x}) - \ve{x}^*\| \ge d(\ve{x^*},H_i) = |\epsilon_i| \ge \epsilon^* > \|\ve{x} - \ve{x}^*\|.$
	
	For $i \not\in I$, since $\ve{x}^* \in H_i$, $$r_i = d(\ve{x},H_i) \le \|\ve{x} - \ve{x}^*\| < \frac{1}{2} \epsilon^*.$$  To summarize, 
	$$|r_i| = |\ve{a}_i^T \ve{x}_k - b_i| \begin{cases} < \frac{1}{2} \epsilon^* &\text{for } i \not\in I \\ > \frac{1}{2} \epsilon^* &\text{for } i \in I \end{cases}.$$
	
	Thus, if we consider the above, $D = \underset{D \subset [A], |D| = d}{\text{argmax}} \; \underset{i \in D}{\sum} |A\ve{x} - \ve{b}|_i$ is clearly a subset of $I$ for $d \le s$.
\end{proof}

\begin{proof}[Proof of Lemma \ref{lem:probbound}]
	Let $E$ be the event that $i_1, i_2, ..., i_{k^*} \not\in I$ for all index selections in round $W$.  Note that $$\mathbb{P}(E) \ge \bigg(\frac{m-s}{m}\bigg)^{k^*}$$ since there are $m - s$ consistent equations and the equations are being selected uniformly at random.
	
	Now, note that if one conditions upon $E$ and looks at the expected value of $\|\ve{x}_{k^*} - \ve{x}^*\|^2$, this will be the same value as the expectation of $\|\ve{x}_{k^*} - \ve{x}^*\|^2$ if $\ve{x}_{k^*}$ is created with randomized Kaczmarz run on $A_*, \ve{b}_*$; we denote this expectation as $\mathbb{E}_{A_*,\ve{b}_*}[\|\ve{x}_{k^*} - \ve{x}^*\|^2]$.  Applying Theorem \ref{strohmervershynin}, we see that
	\begin{align*}
	\mathbb{E}[\|\ve{x}_{k^*} - \ve{x}^*\|^2 | E] &= \mathbb{E}_{A_*,\ve{b}_*}[\|\ve{x}_{k^*} - \ve{x}^*\|^2]
	\\&\le \bigg(1 - \frac{\sigma_{\min}^2(A_*)}{m-s}\bigg)^{k^*} \|\ve{x}_0 - \ve{x}^*\|^2
	\\&= \bigg(1 - \frac{\sigma_{\min}^2(A_*)}{m-s}\bigg)^{k^*} \|\ve{x}^*\|^2.
	\end{align*}
	Now, since $k^* \ge \frac{\log\Big(\frac{\delta(\epsilon^*)^2}{4\|\ve{x}^*\|^2}\Big)}{\log\Big(1 - \frac{\sigma_{\min}^2(A_*)}{(m-s)}\Big)}$, we have $\Big(1 - \frac{\sigma_{\min}^2(A_*)}{m-s}\Big)^{k^*} \le \frac{\delta(\epsilon^*)^2}{4\|\ve{x}^*\|^2}$ and so $$\mathbb{E}[\|\ve{x}_{k^*} - \ve{x}^*\|^2 | E] \le \frac{\delta}{4} (\epsilon^*)^2.$$
	
	Applying the conditional Markov inequality, we have 
	\begin{align*}
	\mathbb{P}[\|\ve{x}_{k^*} - \ve{x}^*\|^2 > \frac{1}{4}(\epsilon^*)^2 | E] &\le \frac{\mathbb{E}[\|\ve{x}_{k^*} - \ve{x}^*\|^2 | E]}{\frac{1}{4}(\epsilon^*)^2}
	\\&\le \frac{\frac{\delta}{4} (\epsilon^*)^2}{\frac{1}{4}(\epsilon^*)^2} = \delta
	\end{align*}
	
	Thus, $\mathbb{P}[\|\ve{x}_{k^*} - \ve{x}^*\|^2 \le \frac{1}{4}(\epsilon^*)^2 | E] \ge 1 - \delta$ so $$\mathbb{P}[\|\ve{x}_{k^*} - \ve{x}^*\| \le \frac{1}{2}\epsilon^*] \ge (1 - \delta)\bigg(\frac{m-s}{m}\bigg)^{k^*}.$$
\end{proof}

\begin{proof}[Proof of Theorem \ref{thm:d=Iwindowedprob}]
	Since $d \ge s$, we need only have one `successful' round where $\|\ve{x}_{k^*} - \ve{x}^*\| < \frac{1}{2} \epsilon^*$ in order to guarantee detection of all of the corrupted equations, by Lemma \ref{lem:detectionhorizon}.  Since all of the rounds are independent from each other in Method \ref{WKwoR} and by Lemma \ref{lem:probbound} the probability that $\|\ve{x}_{k^*} - \ve{x}^*\| < \frac{1}{2} \epsilon^*$ is at least $p := (1-\delta)\Big(\frac{m-s}{m}\Big)^{k^*}$, we may bound the probability of success by that of a binomial distribution with parameters $W$ and $p$.  Thus, success happens with probability at least $$1 - \bigg[1 - (1-\delta)\Big(\frac{m-s}{m}\Big)^{k^*}\bigg]^{W}.$$
\end{proof}

\begin{proof}[Proof of Theorem \ref{thm:dsmallerthanIwindowedprob}]
	Since $d \ge 1$ and we are selecting unique indices in each iteration of Method \ref{WKwoRUS}, we need to have $\lceil s/d \rceil$ `successful' rounds where $\|\ve{x}_{k^*} - \ve{x}^*\| < \frac{1}{2} \epsilon^*$ in order to guarantee detection of all of the corrupted equations, by Lemma \ref{lem:detectionhorizon}.  Since all of the rounds of RK iterations are independent from each other in Method \ref{WKwoRUS} and by Lemma \ref{lem:probbound} the probability that $\|\ve{x}_{k^*} - \ve{x}^*\| < \frac{1}{2} \epsilon^*$ is at least $p:= (1-\delta)\Big(\frac{m-s}{m}\Big)^{k^*}$, we may bound the probability of success by that of a cumulative binomial distribution with parameters $W$ and $p$ and we calculate the probability that the number of successes, $j \ge \lceil s/d \rceil$.  Thus, success happens with probability defined by the probabilities that less than $\lceil s/d \rceil$ rounds are successful.  The probability of success is bounded below by $$1 - \sum_{j=0}^{\lceil s/d \rceil - 1} {W \choose j} p^j (1-p)^{W-j}.$$
\end{proof}

\end{document}